\documentclass{article}

\usepackage{amsmath,amsthm,indentfirst}
\usepackage{amssymb}
\usepackage{amsfonts}
\usepackage{amscd}
\usepackage[latin1]{inputenc}

\theoremstyle{plain}
\newtheorem*{maintheorem*}{Main Theorem}
\newtheorem*{thm*}{Theorem}
\newtheorem*{thma*}{Theorem A}
\newtheorem*{thmaa*}{Theorem A'}
\newtheorem*{thmb*}{Theorem B}
\newtheorem*{thmo*}{Theorem 1.1}
\newtheorem*{thmc*}{Theorem C}
\newtheorem*{thmd*}{Theorem D}
\newtheorem*{thmf*}{Theorem 4.1}
\newtheorem*{remark*}{Remark}
\newtheorem*{conjecture*}{Conjecture}
\newtheorem*{prop*}{Proposition}
\newtheorem{thm}{Theorem}[section]
\newtheorem{cor}[thm]{Corollary}
\newtheorem{lem}[thm]{Lemma}

\theoremstyle{definition}

\newtheorem*{proofc*}{Proof of Theorem C}

\newtheorem{problem}[thm]{Problem}
\newtheorem{definition}[thm]{Definition}

\newtheorem{remark}[thm]{Remark}

\def\bbz{\mathbb{Z}}
\def\bbq{\mathbb{Q}}
\def\bbf{\mathbb{F}}
\def\bbr{\mathbb{R}}

\def\bbc{\mathbb{C}}

\def\bbn{\mathbb{N}}

\def\bfr{\mathfrak{B}}
\def\pfr{\mathfrak{P}}
\def\ybf{\mathbf{y}}
\def\bbf{\mathbf{B}}
\def\dbf{\mathbf{D}}
\def\wbf{\mathbf{w}}
\def\vare{\varepsilon}

\def\e{\mathbf{e}}
\def\l{|\langle}
\def\r{\rangle|}
\def\d{\partial}
\def\g{\nabla}
\def\p{\prod_{\nu\in S}}

\def\u{\mathcal{U}}
\def\U{\mathbf{U}}
\def\x{\mathbf{x}}
\def\fbf{\mathbf{f}}
\def\diag{{\rm diag}}
\def\cf{\check{f}}

\def\GL{\rm{GL}}
\def\rank{\rm{rank}}
\def\q{\mathbf{q}}
\def\para{\|}
\def\h{\hspace{1mm}}
\def\ball{\prod_{\nu\in S} B(x_{\nu},\frac{1}{4d_{\infty}\|\g{g_{\nu}}(x_{\nu})\|_{\nu}})}
\def\fball{\prod_{\nu\in {S_{\mathcal{R}}}^c} B(x_{\nu},\frac{1}{4R^{1/\widetilde{\kappa_{\mathcal{R}}}}\|\g{g_{\nu}}(x_{\nu})\|_{\nu}})}

\def\infball{\prod_{\nu\in S_{\mathcal{R}}} B(x_{\nu},\frac{R^{1/\kappa_{\mathcal{R}}}}{4d_{\infty}\|\g g_{\nu}(x_{\nu})\|_{\nu}})}

\title{\sc $S$-Arithmetic Khintchine-Type Theorem}
\author{A.~Mohammadi
\and
A.~Salehi Golsefidy\footnote{A. S-G. was partially supported by the NSF grant DMS-0635607. Part of the research conducted while A. S-G. was a Liftoff fellow.  }}
\date{1/22/2009}

\begin{document}
%%%%%%%%%%%%%%%%%%%%%%%%%%%%%% Abstract %%%%%%%%%%%%%%%%%%%%%%%%%%%%%%%%%
\maketitle
\begin{abstract}

\noindent
In this article we prove a convergence $S$-arithmetic Khintchine-type theorem for product of non-degenerate $\nu$-adic manifolds, where one of them is the  Archimedian place. 
\end{abstract}

%%%%%%%%%%%%%%%%%%%%%%%%%%%%%% Introduction %%%%%%%%%%%%%%%%%%%%%%%%%%%%%
\section{Introduction}

\textbf{Metric Diophantine approximation on $\bbr ^n$.}
Any real number can be approximated by rational numbers.
Diophantine approximation concerns the precision of the approximation. For instance, by 
Dirichlet's Theorem, one can see that for any real number $\xi$, there are infinitely many integers $p$ and
$q$, such that $|q\xi-p|<1/q$. This, in some sense,  indicates that in order to get a ``good" approximation you do
not need a ``very large" denominator. On the other hand, it is well known that any quadratic algebraic number 
cannot be ``very well" approximated. One can ask what happens for a ``random" number, which 
is the subject of metric Diophantine approximation. Let us be more precise. Let $\psi$ be a 
decreasing function from $\bbr^+$ to $\bbr^+$ e.g. $\psi_{\vare}(q)=1/q^{1+\vare}$. A real number 
$\xi$ is called {\it $\psi$-approximable} if for infinitely many integers $p$ and $q$, 
one has $|q\xi-p|<\psi(|q|)$. It is called {\it very well approximable} if it is $\psi_{\vare}$-A for 
some positive $\vare$. A.~Khintchine~\cite{Kh24} has shown that almost all (resp. almost no) points, in terms of the 
Lebesgue measure, are $\psi$-A if $\sum_{q=1}^{\infty}\psi(q)$ diverges 
(resp. converges)(See~\cite[Chapter IV, Section 5]{Sch}) . There are two ways to generalize the notion of $\psi$-A to $\bbr^n$:
\begin{itemize}\item[a)] $\|q{\bf\xi}-{\bf p}\|<\psi(|q|)^{1/n}$ for infinitely many $q\in\bbz$ and
 ${\bf p}\in \bbz^n$.
\item[b)]$|{\bf q}\cdot{\bf \xi}-p|<\psi(\|{\bf q}\|^n)$ for infinitely many ${\bf q}\in\bbz^n$ and
 $p\in\bbz$.\end{itemize}
 
A priori there are two notions of VWA vectors, i.e. being $\psi_{\vare}$-A for some positive $\vare,$
either in terms of (a) or (b). However by means of Khintchine transference principle, these two give
rise to the same notion (See~\cite[Chap. 1]{BD}) or~\cite[Chap. 5]{Ca}).  Groshev~\cite{Gr} proved the aforementioned theorem in setting (b) (See~\cite{Gr}), while in setting (a) it was already known to Khintchine in 1926. 

\textbf{Metric Diophantine approximation on manifolds.}
One can 
restrict oneself to a submanifold of $\bbr^n$, and ask if a random point on this submanifold is
 $\psi$-A. In fact one of the first questions in this direction was posed by K.~Mahler~\cite{Ma}. He conjectured
 that almost all points on the {\it Veronese curve} $\{(x,x^2,\cdots,x^n)|x\in\bbr\}$ are not VWA. Lots of works had been done to prove this conjecture by J.~Kubilius, B.~Volkmann, W.~LeVeque, F.~Kash, and W.~M.~Schmidt. In particular, the problem was solved for $n=2$ by Kubilius~\cite{Ku} and for $n=3$ by Volkmann~\cite{Vol}. Finally Mahler's conjecture was settled affirmatively by Sprind\v{z}uk~\cite{sp1,sp2}, and his proof led to the theory of Diophantine approximation
 on manifolds. According to his terminology, a submanifold $M\subseteq \bbr^n$ is called {\it extremal}
 if almost all points of $M$ are not VWA. He conjectured~\cite{sp3} that any ``nondegenerate" submanifold of
 $\bbr^n$ is extremal (ref.~\cite{BD} for the definition of nondegeneracy). In fact, he conjectured this in the analytic setting. It is worth mentioning that a manifold $M=\{(f_1({\bf x}),\cdots,f_n({\bf x}))|{\bf x}\in U\}$ with analytic coordinates $f_i$'s is non-degenerate if and only if the functions $1,f_1,\cdots,f_n$ are linearly independent over $\bbr$. D.~Kleinbock and G.~Margulis~\cite{KM} proved a stronger version of this
 conjecture using dynamics of special unipotent flows on the space of lattices. Later V.~Bernik, 
 D.~Kleinbock, and G.~Margulis~\cite{BKM} and V.~Beresnevich~\cite{Ber1,Ber2,Ber3} independently proved a convergence Khintchine-type theorem on manifolds. 
 For instance, they showed that if $\sum_{{\bf q}\in \bbz^n\setminus{\bf 0}}\psi(\|{\bf q}\|^n) $ 
 converges, almost no point of a non-degenerate submanifold is $\psi$-A. The divergence case has
 been also completely solved by V.~Beresnevich, V.~Bernik, D.~Kleinbock, and G.~Margulis~\cite{BBKM}.

\textbf{$S$-arithmetic Diophantine approximation.}
There are relatively less known results in the $p$-adic, and simultaneous approximation in different places. In a recent work V.~Beresnevich, V.~Bernik, E.~Kovalevskaya~\cite{BBK}, proved both the convergence and the divergence Khintchine-type theorem for the $p$-adic  Veronese curve, i.e. $\{(x,x^2,\cdots,x^n)|x\in\bbq_p\}$. It is worth mentioning that the convergence case had been already proved by E.~Kovalevskaya~\cite{Ko1}. There are a few other results of convergence Khintchine-type  for more general curves in the $p$-adic plane or space, e.g.~\cite{BK1,Ko2}. We take on this case in~\cite{MS} where we prove both the convergent and the divergent Khintchine-type theorem for non-degenerate $p$-adic manifolds.

Situation in the simultaneous Diophantine approximation is even less clear. The most general Khintchine-type theorem, in this case,  is a recent work of   V.~Bernik and E.~Kovalevskaya\cite{BK2}. They establish an inhomogeneous convergence Khintchine-type theorem for the Veronese curve with components in product of several local fields, more specifically  $\{(\x,\x^2,\dots,\x^n)|\x\in\bbc\times\bbr\times\prod_{p\in S}\bbq_p\}$. For product of non-degenerate manifolds, D.~Kleinbock and G.~Tomanov, in a recent paper~\cite{KT}, came up with an $S$-arithmetic version of metric Diophantine approximation. They carefully defined the notion of extremal manifolds and proved the analogous theorem. Let us briefly recall some of the definitions and results from their work. 

Fix a set $S$ of cardinality $\kappa$ consisting of distinct normalized valuations of $\bbq$. Let $\bbq_S=\p \bbq_{\nu}, S_f=S\setminus \{\nu_{\infty}\}$, and $\tilde{S}=S\cup \{\nu_{\infty}\}$. Using a Dirichlet-principle-type argument, one can show that for any $\xi\in\bbq_S^n$ with $S$-norm at most one, there are infinitely many $(\q,q_0)\in\bbz^n\times\bbz$, such that 
$$ |\q\cdot\xi+q_0|_S^{\kappa}\le \|\q\|_{\infty}^{-n}\hspace{1cm}\mbox{if}\hspace{.2cm}\nu_{\infty}\in S,$$
$$|\q\cdot\xi+q_0|_S^{\kappa}\le \|(\q,q_0)\|_{\infty}^{-n-1}\hspace{1cm}\mbox{if}\hspace{.2cm}\nu_{\infty}\not\in S.$$
Accordingly, they defined the notion of VWA for a vector in $\bbq_S^n$; namely, $\xi\in\bbq_S^n$ is called VWA if for some $\vare>0$, there are infinitely many $(\q,q_0)\in\bbz^n\times\bbz$, such that
$$ |\q\cdot\xi+q_0|_S^{\kappa}\le \|\q\|_{\infty}^{-n-\vare}\hspace{1cm}\mbox{if}\hspace{.2cm}\nu_{\infty}\in S,$$
$$|\q\cdot\xi+q_0|_S^{\kappa}\le \|(\q,q_0)\|_{\infty}^{-n-1-\vare}\hspace{1cm}\mbox{if}\hspace{.2cm}\nu_{\infty}\not\in S.$$
Extremal submanifolds of $\bbq_S^n$ were defined similar to the real case and they proved that:
\begin{thma*}
Let $M\subseteq\mathbb{Q}_S^n$ be a non-degenerate $C^k$ manifold, then $M$ is extremal.
\end{thma*} 
\textbf{A few terminologies and the statement of the main result.}
Here we introduce necessary notations to state the main results of the article, and refer the reader to the second section for the definitions of the technical terms. Let $S,S_f$ and $\tilde{S}$ be as before. It is well known  that $\bbz_{\tilde{S}}=\bbq\cap\bbq_{\tilde{S}}\cdot\prod_{\nu\not\in {\tilde{S}}}\bbz_{\nu}$ is a co-compact lattice
 in $\bbq_{\tilde{S}}$, and $[0,1)\times \prod_{\nu\in S_f}\bbz_{\nu}$ is a fundamental domain of $\bbz_{\tilde{S}}$ in $\bbq_{\tilde{S}}$. As we mentioned before any vector in $\bbq_S^n$ can be approximated with rational vectors. However this time, we view the field of rational numbers as the field of fractions of $\mathcal{R}$ a subring of $\bbz_{\tilde{S}}$. It is worth mentioning that any subring of $\bbz_{\tilde{S}}$ is of the form $\bbz_{S'}$ for a subset $S'$, which contains the infinite place, of $\tilde{S}$. Any such $\mathcal{R}$ is discrete, so it has just finitely many elements $a_r$ in a ball of radius $r$ in $\bbq_{\tilde{S}}$. It is easy to see that $a_r$ grows polynomially with the growth degree $g(\mathcal{R})$ equal to $|S'|$. In particular one has $|B_r\cap \mathcal{R}|<2 r^{g(\mathcal{R})}$.

Using Dirichlet-principle-type argument one can see that for any $\xi\in \bbq_S^n$ with $S$-norm at most one, there are infinitely many $(\q,q_0)\in \mathcal{R}^n\times\mathcal{R}$ such that
$$ |\q\cdot\xi+q_0|_S^{\kappa}\le \|\q\|_{S}^{-ng(\mathcal{R})}\hspace{1cm}\mbox{if}\hspace{.2cm}\nu_{\infty}\in S,$$
$$|\q\cdot\xi+q_0|_S^{\kappa}\le \|(\q,q_0)\|_{\tilde{S}}^{-(n+1)g(\mathcal{R})}\hspace{1cm}\mbox{if}\hspace{.2cm}\nu_{\infty}\not\in S.$$
Accordingly, one can define the notion of $\mathcal{R}$-VWA; namely, $\xi\in\bbq_S^n$ is called $\mathcal{R}$-VWA if for some $\vare>0$, there are infinitely many $(\q,q_0)\in\mathcal{R}^n\times\mathcal{R}$ such that
$$ |\q\cdot\xi+q_0|_S^{\kappa}\le \|\q\|_{S}^{-ng(\mathcal{R})-\vare}\hspace{1cm}\mbox{if}\hspace{.2cm}\nu_{\infty}\in S\hspace{.3cm},$$
$$|\q\cdot\xi+q_0|_S^{\kappa}\le \|(\q,q_0)\|_{\tilde{S}}^{-(n+1)g(\mathcal{R})-\vare}\hspace{1cm}\mbox{if}\hspace{.2cm}\nu_{\infty}\not\in S.$$
A manifold $M\subseteq\bbq_S^n$ is called $\mathcal{R}$-extremal if almost no point of it is $\mathcal{R}$-VWA.

One can rephrase the result of Kleinbock and Tomanov, theorem A, and say that any non-degenerate $C^k$ manifold is $\bbz$-extremal. In fact, it is easy to adapt their proof and show that any such manifold is $\mathcal{R}$-extremal, for any subring $\mathcal{R}$ of $\bbz_{\tilde{S}}$. 

Assume that $S$ contains the Archimedian place. Let $\Psi$ be a map from $\mathcal{R}^n$ to $\bbr^+$. A vector $\xi\in \bbq_S^n$ is called $(\Psi,\mathcal{R})$-A if for infinitely many $(\q,q_0)\in\mathcal{R}^n\times\mathcal{R}$ one has $|\q\cdot\xi+q_0|_S^{\kappa}\le\Psi(\q).$
In this article, we prove an \textit{$S$-arithmetic $\mathcal{R}$-Khintchine-type} statement. Let us fix a few notations before stating the precise statements.
\\

\noindent
{\bf a) Places:} $S$: a finite set of places containing the Archimedian place.
\\

\noindent
{\bf b-1) Domain:} $\mathbf{U}=\p U_{\nu}$ where $U_{\nu}\subseteq \bbq_{\nu}^{d_{\nu}}$ is an open box.
\\

\noindent
{\bf b-2) Coordinates:} $\mathbf{f}(\x)=(f_{\nu}(x_{\nu}))_{\nu\in S}$, for any $\x=(x_{\nu})\in
\mathbf{U}$, where

\begin{itemize}\item[i)] $f_{\nu}=(f_{\nu}^{(1)},\cdots,f_{\nu}^{(n)}): U_{\nu} \rightarrow
\bbq_{\nu}^n$: analytic map for any $\nu\in S$, and can be analytically extended to the boundary of $U_{\nu}$.
\item[ii)] Restrictions of $1,\hspace{.5mm} f_{\nu}^{(1)},\cdots,f_{\nu}^{(n)}$ to any open subset of $U_{\nu}$ are linearly independent over $\mathbb{Q}_{\nu}$. 
\item[iii)] $\para f_{\nu}(x_{\nu})\para \le 1,$ $\para \g
f_{\nu}(x_{\nu})\para\le 1$, and $|\bar{\Phi}_{\beta}f_{\nu}(y_1,y_2,y_3)|\le 1/2$ for any $\nu \in S$, second difference quotient  $\Phi_\beta$, and $x_{\nu}, y_1,y_2,y_3\in U_{\nu}$ (For the definition of $\Phi_{\beta}$, we refer the reader to the second section). 
\end{itemize}

\noindent
{\bf c) Ring:} $\mathcal{R}$ is a subring of $\bbz_S$, and so for some $S_{\mathcal{R}}\subseteq S$, we have $\mathcal{R}=\bbz_{S_{\mathcal{R}}}$. Let ${S_{\mathcal{R}}}^c$ be the complement of $S_{\mathcal{R}}$ in $S$.
\\

\noindent
{\bf d) Level of approximation:} $\Psi:\mathcal{R}^n\setminus\{0\}\rightarrow(0,\infty)$ satisfies
\begin{itemize}
\item[i)]$\Psi(q_1,\cdots,q_i,\cdots,q_n)\ge\Psi(q_1,\cdots,q_i',\cdots,q_n)$
whenever $|q_i|_S\le|q_i'|_S$.
\item[ii)]$\sum_{\mathbf{q}\in\mathcal{R}^n\setminus\{0\}}\Psi(\mathbf{q})<\infty$.
\end{itemize}

\begin{thm}~\label{khintchine}
Let $S$, $\U,$ $\fbf,$ $\Psi$, and $\mathcal{R}$ be as above; then $$\mathcal{W}_{\mathcal{R},\Psi}^{\hspace{1mm}\fbf}=\{\x\h|\h\fbf(\x)\h\mbox{is}\h (\Psi,\mathcal{R})-A\}$$
has measure zero.
\end{thm}

\begin{remark*}

\begin{itemize}
\item[1-]This theorem has been proved when $S=\{\nu_{\infty}\}$ by Bernik, Kleinbock and Margulis~\cite{BKM}.
\item[2-]As we mentioned earlier, for $\Psi(\q)=\|\q\|_S^{-ng(\mathcal{R})-\vare}$, where $\vare$ is a positive number, Kleinbock and Tomanov~\cite{KT} essentially proved this theorem.
\item[3-]Kleinbock and Tomanov~\cite{KT} asked for such a theorem for $\mathcal{R}=\bbz$.
\item[4-]It is clear that the above theorem holds for product of any non-degenerate $\nu$-adic analytic manifolds. Indeed, the condition on the domain or the first two conditions on the coordinate functions are consequences of analyticity and non-degeneracy of the manifold. The last condition on the coordinate functions can be achieved by replacing $\fbf$ with $\fbf/M$ for a suitable $S$-integer $M$. 
\item[5-](With or without the Archimedian place) As we have seen in the introduction, there is an intrinsic difference between the cases either with or without Archimedian places even though the methods are somewhat similar. For instance in the setting of this paper, namely when $\nu_{\infty}\in S$, we formulate and prove a simultaneous approximation with coordinates in any subring $\mathcal{R}$. However when $\nu_{\infty}\not\in S,$ we can formulate and prove such a theorem only for $\mathcal{R}=\bbz$, in~\cite{MS}, see the remarks at the end of this paper for the precise statement. 
\item[6-](Divergence) Following the above remark, we should also mention that in~\cite{MS} we also prove the divergence counter part as long as we deal with only one non-Archimedian place. Our argument comes short of proving the divergence counter part of simultaneous approximations. In particular, in the setting of this paper, namely when $S$ contains the Archimedian place and a non-Archimedian place, we do not get the divergence part.  
\end{itemize}

\end{remark*}
 
{\bf Idea of the proof of theorem~\ref{khintchine}.} 
We essentially follow the same stream line of the 
proof of~\cite{BKM}. However we have to do careful analysis on families of $p$-adic $C^k$ functions. Spaces over $p$-adic fields are ``easier" when one deals with number theoretic properties. However the analysis in some problems gets subtle as we have neither {\it the notion of angle} nor {\it connectedness}! So almost all the steps need a new approach or at least perspective.

For a fixed $\q\in\mathcal{R}^n\setminus 0$ we study the behavior of the function $\fbf(\x)\cdot\q$, and the philosophy is that when the gradient of this function is ``large", the value of the function cannot be  close to $\mathcal{R}$  for a ``long" time. This will take care of points with ``large" gradient. Hence we need to deal with the points with ``small" gradient. To do so, we use a quantitative version of recurrence of polynomial-like maps on the space of $S$-arithmetic modules. What is roughly explained here is the process of reducing the proof of theorem~\ref{khintchine} to the following theorems, modulo Borel-Cantelli Lemma.  

\begin{thm}~\label{>}
Let $\mathbf{U}$ and $\hspace{1mm}\mathbf{f}$ be as above and $0<\epsilon< \frac{1}{4n|S|(1+|{S_{\mathcal{R}}}^c|)}.$  Let $\mathcal{A}$ be
 $$\left\{\x\in\U|\h\exists\hspace{.5mm}\q\in\mathcal{R}^n,\h \frac{T_i}{2}\leq\h|q_i|_{S}<T_i,\hspace{.5mm}\begin{array}{l}\l\mathbf{f}(\x)\cdot\q\r_{S}^{|S|}<\delta(\prod_i T_i)^{-g(\mathcal{R})}\\
\|\g f_{\nu}(x_{\nu})\q\|_{\nu}>\|\q\|_S^{-\epsilon},\hspace{4mm}\nu\in {S_{\mathcal{R}}}^c\h \\ \|\g f_{\nu}(x_{\nu})\q\|_{\nu}>\|\q\|_S^{1-\epsilon},\hspace{3mm}\nu\in S_{\mathcal{R}}\end{array}\right\};$$ then
 $|\mathcal{A}|<C \delta\hspace{1mm}|\U|,$ for large enough $\max (T_i)$ and a universal constant $C$. 
\end{thm}

For the convenience of the reader, let us recall that $\mathcal{R}=\bbz_{S_{\mathcal{R}}}$, and the growth of the numbers of elements of $\mathcal{R}$ in a ball of radius $T$ in $\bbq_S$ is a polynomial on $T$ of degree $g(\mathcal{R})=|S_{\mathcal{R}}|$. Let us also add that whenever needed we view a vector as a column or a row matrix.

\begin{thm}~\label{<}
Let $\U$ and $\mathbf{f}$ be as before. If $\| \fbf(\x)\|\le 1$ and $\|\g\fbf(\x)\|\le 1$, then for any
$\x=(x_{\nu})_{\nu\in S}\in \U$, one can find a neighborhood
$\mathbf{V}=\p V_{\nu}\subseteq \U$ of $\x$ and $\alpha>0$ with
the following property: For any ball $\mathbf{B}\subseteq \mathbf{V}$,
there exists $E>0$ such that for any choice of $0<\delta\le 1$,
$T_1,\cdots,T_n\ge 1$, and $K_{\nu}>0$ with $\delta^{|S|}{ (\frac{T_1\cdots
T_n}{\max_i T_i})}^{g(\mathcal{R})}\p K_{\nu}\le 1$ one has
$$\left|\left\{\x\in\mathbf{B}|\hspace{1mm}\exists
\q\in\mathcal{R}^n\setminus\{0\}:\begin{array}{l}\l\mathbf{f}(\x)\cdot\q\r<\delta\\
\|\g f_{\nu}(x)\q\|_{\nu}<K_{\nu},\hspace{2mm}\nu\in S\\|q_i|_S<T_i\end{array}\right\}\right|\le
E\hspace{.5mm}\vare^{\alpha}|\mathbf{B}|,\hspace{5mm}~(\ref{<})$$ where
$\vare=\max\{\delta,(\delta^{|S|}{ (\frac{T_1\cdots
T_n}{\max_i T_i})}^{g(\mathcal{R})}\p K_{\nu})^{\frac{1}{\kappa(n+1)}}\}.$
\end{thm}

Theorem~\ref{<}  is proved using dynamics of special unipotent flows and $S$-arithmetic version of Kleinbock-Margulis lemma, which was proved in~\cite{KT}.

{\bf Structure of the paper.}
In section 2, we start with some geometry and analysis of $p$-adic spaces, and continue observing some of the properties of discrete $\bbz_S$-submodules of $\p \bbq_{\nu}^{m_{\nu}}$. Section 3 is devoted to the proof of theorem~\ref{>}. In section 4, we recall the notion of good functions and establish the ``goodness" of families of $\nu$-adic analytic functions, which will be needed in the proof of theorem~\ref{<}. This technical section, in some sense, is the core of the proof of theorem~\ref{<} modulo theorem~\ref{poset}. In section 5, we translate theorem~\ref{<} in terms of recurrence of special flows on the space of discrete $\bbz_S$-submodules of $\p \bbq_{\nu}^{m_{\nu}}$. In section 6, we shall recall a theorem of  Kleinbock-Tomanov, and use it to establish theorem~\ref{<} proving its equivalence in the dynamical language. The proof of the main theorem will be completed in section 7. We shall finish the paper by discussing a few remarks, and open problems.

{\bf Acknowledgments.}
Authors would like to thank G.~A.~Margulis for introducing this topic and suggesting this problem to them. We are also in debt to D.~Kleinbock for reading the first draft and useful discussions. We also thank the anonymous referee(s) for their remarks and suggestions. 

%%%%%%%%%%%%%%%%%%%%%%%%%%%%%%%%%%%%%%%%%%%%%%%%%%%%%%%%%%%%%%%%%%%%%%%%%%%
\section{Notations and Preliminaries}
%%%%%%%%%%%%%%%%%%%%%%%%%% Orthogonal basis %%%%%%%%%%%%%%%%%%%%%%%%%%%%%%%%
\textit{\textbf{Geometry of $p$-adic spaces.}} For any place $\nu$
of the field of rational numbers $\bbq$, $\bbq_{\nu}$ denotes the
$\nu$-completion of $\bbq$. In particular, when $\nu$ is the
Archimedian place of $\bbq$, $\bbq_{\nu}$ is the field of real
numbers $\bbr$. A non-Archimedian place (resp. the Archimedian
place) is also called a finite (resp. infinite) place. Let $p_{\nu}$ be the number of elements of the residue field of $\bbq_{\nu}$ if $\nu$ is a finite place. For $a$ a positive real
number and $\nu$ a finite place of $\bbq$, let $\lceil
a\rceil_{\nu}$ (resp. $\lfloor a\rfloor_{\nu}$) denote a power of
$p_{\nu}$ with the smallest (resp. largest) $\nu$-adic norm bigger
(resp. smaller) than $a$. For any ring $\mathcal{R}$ and two
vectors $x,y\in \mathcal{R}^n$, we set $x\cdot y=\sum_{i=1}^n
x^{(i)}y^{(i)}$. Let $\nu$
be a place of $\bbq$ and $\mathcal{V}$ a vector space over
$\bbq_{\nu}$.
 For any subset $\mathcal{X}$ of $\mathcal{V}$,
$\mathcal{X}_{\bbq_{\nu}}$ (resp. $\mathcal{X}_{\bbz_{\nu}}$) denotes the $\bbq_{\nu}$ (resp. $\bbz_{\nu}$) span of $\mathcal{X}$. We recall the notion of
distance and orthogonality on $\mathcal{V}$ even if $\nu$ is a finite place. In the infinite place we take the Euclidean
norm on $\mathcal{V}$, and in a finite place, for the notion of
distance, we take a $\bbq_{\nu}$ basis $\bfr$ for
$\mathcal{V}$, and define the maximum norm
$\para\cdot\para_{\bfr}$ with respect to $\bfr$ on $\mathcal{V}$.
For the space $\bbq_{\nu}^{m}$, one can consider the norm with
respect to the standard basis, and in this case we drop $\bfr$ from the notation. Any basis for
$\mathcal{V}$ gives rise to a basis for $\bigwedge \mathcal{V}$,
so we can extend the corresponding norm on $\mathcal{V}$ to a norm on
$\bigwedge \mathcal{V}$, and we do so. The following definition and/or lemma gives us the notion of orthogonality. 

\begin{definition}~\label{orthogonal}
Let $\nu$ be a finite place of $\bbq$. A set of vectors ${x_1, \cdots, x_n}$
in $\bbq _{\nu}^m$, is called orthonormal if  $ \|x_1\| = \|x_2\| = \cdots =
\|x_n\| = \|x_1 \wedge \cdots \wedge x_n\| = 1$, or equivalently when it
can be extended to a $\bbz_{\nu}$-basis of $\bbz_{\nu}^m$.
\end{definition}

%%%%%%%%%%%%%%%%%%%%%%%%%%%% p-adic calculus %%%%%%%%%%%%%%%%%%%%%%%%%%%%%%%%%%%%%%
\textbf{\textit{Calculus of functions on $p$-adic spaces.}} Here
 we recall the definition of $p$-adic $C^k$ functions, and refer the reader to \cite{Sc} for further reading. Let $F$ be a local field and $f$ an $F$-valued function defined
on $U$ an open subset  of $F$. The first difference quotient
$\Phi^1 f$ of $f$ is a two variable function given by
$$\Phi^1 f(x,y):=\frac{f(x)-f(y)}{x-y}
\hspace{2mm},$$ defined on $\nabla^2 U:=\{(x,y)\in U\times
U|\hspace{1mm}x\neq y\}.$ We say, $f$ is $C^1$ at $a\in U$ if
$$\lim_{(x,y)\rightarrow (a,a)}\hspace{2mm}\Phi^1 f(x,y)$$ exists,
and $f$ is said to be $C^1$ on $U,$ if it is $C^1$ at every
point of $U.$ Now let $$\nabla^k U:=\{(x_1,\cdots,x_k)\in
U^k|\hspace{1mm}x_i\neq x_j
\hspace{1mm}\mbox{\rm{for}}\hspace{1mm} i\neq j\},$$ and define
the $k^{th}$ order difference quotient $\Phi^k f:
\nabla^{k+1}U\rightarrow F$ of $f$ inductively by $\Phi^0 f= f$
and $$\Phi^k
f(x_1,x_2,\cdots,x_{k+1}):=\frac{\Phi^{k-1}f(x_1,x_3,\cdots,x_{k+1})-\Phi^{k-1}f(x_2,x_3,\cdots,x_{k+1})}{x_1-x_2}.$$
One readily sees any other pair could be taken instead of
$(x_1,x_2)$, and so $\Phi^k f$ is a symmetric function of its
$k+1$ variables.  $f$ is called $C^k$ at $a\in U$ if the following
limit exits
$$\lim_{(x_1,\cdots,x_{k+1})\rightarrow(a,\cdots,a)} \Phi^k
f(x_1,\cdots,x_{k+1}),$$ and it is called $C^k$ on $U$ if it is $C^k$
at every point $a\in U$. This is equivalent to $\Phi^k f$
being continuously  extendable to $\bar{\Phi}^k f:U^{k+1}\rightarrow F.$ Clearly the continuous extension is unique. As one expects $C^k$ functions are $k$ times
differentiable, and
$$f^{(k)}(x)=k!\bar{\Phi}^k(x,\cdots,x).$$ It is worth mentioning
that, $f\in C^k$ implies $f^{(k)}$ is continuous but the converse
fails. Also $C^{\infty}(U)$ is defined to be the class of
functions which are $C^k$ on $U$, for any $k$. Note that 
analytic functions are $C^{\infty}.$

Now it is straightforward to generalize this to several variables.
Let $f$ be an $F$-valued function defined on $U_1\times\cdots\times U_d$,
where each $U_i$ is an open subset of $F.$ Denote by
$\Phi_{i}^{k} f$ the $k^{th}$ order difference quotient of $f$
with respect to the $i^{th}$ coordinate. Then for any multi-index
$\beta=(i_1,\cdots,i_d)$ let
$$\Phi_{\beta}f:=\Phi_1^{i_1}\circ\cdots\circ\Phi_d^{i_d} f.$$ It
is defined on $\nabla^{i_1+1} U_1\times\cdots\times\nabla^{i_d+1}
U_d$, and as above the order is not important. The function $f$ is called
$C^k$ on $U_1\times\cdots\times U_d$ if for any multi-index $\beta$ with
$|\beta|=\sum_{j=1}^d i_j$ at most $k,\hspace{1mm} \Phi_{\beta} f$
is extendable to $\bar{\Phi}_{\beta} f$ on
$U_1^{i_1+1}\times\cdots\times U_d^{i_d+1}.$ Similarly to the one
variable case the obvious relation between $\bar{\Phi}_{\beta} f$
and $\partial_{\beta} f$ holds.

For a $C^1$ function $f=(f_1,\cdots,f_n)$ from $F^m$ to $F^n$, let $\g f(x)$ be an $m$ by $n$ matrix whose $(i, j)$ entry is $\partial_j f_i(x)$. 

\vspace{1mm}

%%%%%%%%%%%%%%%%%%%%%% Discrete submodules of *Q_v^d_v %%%%%%%%%%%%%%%%%%%%%%%%%%%%
\textbf{\textit{Discrete $\bbz_S$-submodules of $\p
\bbq_{\nu}^{m_{\nu}}$.}}
 For any finite set $S$ of places of $\bbq$ which contains the
  infinite place $\infty$, set $S_f=S\setminus\{\infty\}$, 
  $\bbq_S=\p \bbq_{\nu}$, and $\bbz_S=\bbq\cap(\bbq_S\times\prod_{\nu\not\in S_f}\bbz_{\nu})$
  where $\bbq$ is diagonally embedded in $\bbq_S$. For a non-Archimedian
  (resp. Archimedian) place $\nu$, let us normalize the Haar measure
  $\mu_{\nu}$ of $\bbq_{\nu}$ such that $\mu_\nu(\bbz_{\nu})=1$
   (resp. $\mu_{\infty}([0,1])=1).$
On $\p\bbq_{\nu}^{m_{\nu}}$, we take the maximum norm
$\para\cdot\para_S$, i.e. $\para\x\para_S=\max_{\nu\in S}\para
x_{\nu}\para_{\nu}$. By the Chinese reminder theorem,
   it is clear that $\bbz_S$ is a co-compact lattice in $\bbq_S$,
   and by the above normalization and the product measure
   on $\bbq_S$, the covolume of $\bbz_S$ is one. For any $\x\in \bbq_S$, $\l\x\r$ denotes the distance from $\x$ to $\bbz_S$, and we shall choose $\langle \x \rangle\in \bbz_S$ such that $\para\x-\langle\x\rangle\para_S=\l\x\r$. For any $\x\in\p\bbq_{\nu}^{m_{\nu}}$, let $c(\x)=\p\|x_{\nu}\|_{\nu}$. Here and for all we set $\kappa=|S|,$ clearly one has $c(\x)\leq\|\x\|_S^{\kappa}$. By virtue of~\cite[proposition 7.2]{KT} one can see the following lemma
   which shows that any discrete $\bbz_S$-submodule of $\p \bbq_{\nu}^{m_{\nu}}$ is essentially coming from $\bbz_S$.
\begin{lem}~\label{base}
If $\Delta$ is a discrete $\bbz_S$-submodule of $\p \bbq_{\nu}^{m_{\nu}}$,
 then there are $\x^{(1)},\cdots,\x^{(r)}$ in $\p \bbq_{\nu}^{m_{\nu}}$ so
 that $\Delta=\bbz_S\x^{(1)}\oplus\cdots\oplus\bbz_S\x^{(r)}$. Moreover for
  any $\nu\in S,\hspace{1mm}x_{\nu}^{(1)},\cdots,x_{\nu}^{(r)}$ are linearly independent over $\mathbb{Q}_{\nu}.$
\end{lem}
\begin{definition}~\label{primgood}
Let $\Gamma$ be a discrete $\bbz_S$-submodule
of $\p\bbq_{\nu}^{m_{\nu}}$; then
a submodule $\Delta$ of $\Gamma$ is called a
{\it primitive submodule} if $\Delta=\Delta_{\bbq_S}\cap\Gamma$.
\end{definition}
\begin{remark}~\label{primgood2} Let $\Gamma$ and $\Delta$ be as in
definition~\ref{primgood}, then $\Delta$ is a primitive submodule of $\Gamma$, if and
only if there exists a complementary $\bbz_S$-submodule $\Delta'\subseteq \Gamma$, i.e. $\Delta\cap\Delta'=0$ and $\Delta+\Delta'=\Gamma$.
\end{remark}

%%%%%%%%%%%%%%%%%%%%%%%%%%%%%%%%%%%%%%%%%%%%%%%%%%%%%%%%%%%%%%%%%%%%%%%%%%%%
%%%%%%%%%%%%%%%%%%%%%%%%%% Proof of theorem > %%%%%%%%%%%%%%%%%%%%%%%%%%%%%%
%%%%%%%%%%%

\section{Proof of theorem~\ref{>}}

As in the introduction we have $\mathcal{R}=\bbz_{S_{\mathcal{R}}}$. Let $|S_{\mathcal{R}}|=\kappa_{\mathcal{R}}$ and $|{S_{\mathcal{R}}}^c|=\widetilde{\kappa_{\mathcal{R}}}$ and so $|S|=\kappa=\kappa_{\mathcal{R}}+\widetilde{\kappa_{\mathcal{R}}}$.  
Fix $\q=(q_1,\cdots,q_n)\in\mathcal{R}^n$ with $T_i/2\le |q_i|_S<T_i$ and define $T=\prod_i T_i$ and $R=T^{\frac{1}{n\kappa}}$. As in the theorem, let $0<\epsilon<\frac{1}{4n\kappa(1+\widetilde{\kappa_{\mathcal{R}}})}$ be fixed through out the paper. Let ${\bf g}(\x)=\fbf(\x)\cdot\q$, for any $\x\in {\bf U}$, and set
 \[\mathcal{A}_{\q}=\{\x\in\mathcal{A}|\h\mbox{the hypothesis of the theorem holds for}\h\q =(q_1,\cdots,q_n) \}.\]
 
 \noindent
 This means that any $\x\in \mathcal{A}_{\q}$ satisfies the following properties:
 \begin{itemize}
 \item[P1)] For some $q_0\in \mathcal{R}$, ${|{\bf g}(\x)+q_0|}^{\kappa}_S<\delta T^{-\kappa_{\mathcal{R}}}$.
 \item[P2)] For any $\nu \in {S_{\mathcal{R}}}^c$, $\|{\bf g}(\x)\|_{\nu}>\|\q\|_S^{-\epsilon}$.
 \item[P3)] For any $\nu \in S_{\mathcal{R}}$, $\|{\bf g}(\x)\|_{\nu}>\|\q\|_S^{1-\epsilon}$.
 \end{itemize}
 
 It is worth mentioning that because of (b-2, iii) the third condition on the coordinate maps,  ${\bf g}$ also satisfies the following conditions at any point $\x$:
 
 \begin{itemize}
 \item[C1)] For any $\nu\in S$, $\|\g g_{\nu}(x_{\nu})\|_{\nu} \le \|\q\|_{\nu}$.
 \item[C2)] For any $\nu\in S, 1\le i,j\le d_{\nu}, x_{\nu},x'_{\nu},x''_{\nu}\in U_{\nu}, \bar{\Phi}_{ij}(g_{\nu})(x_{\nu},x'_{\nu},x''_{\nu})\le \|\q\|_{\nu}$.
 \end{itemize}
 \noindent
We will show that, $|\mathcal{A}_\q|<C \delta\h T^{-g(\mathcal{R})}|\U|,$ which then, summing over all possible $\q$'s, will finish the proof. 
\\

\noindent
Let $\bbf(\x)$ be a neighborhood of $\x$ which is defined as follows.

\begin{itemize}
\item[(i)] If $\widetilde{\kappa_{\mathcal{R}}}>0$, let
$$\bbf(\x)=\fball\times\infball$$
\item[(ii)] If $\widetilde{\kappa_{\mathcal{R}}}=0$, let
$$\bbf(\x)=\ball$$
\end{itemize}

For $\x\in\mathcal{A}_\q,$ let $q_0$ be an element in $\mathcal{R}$ which satisfies (P1). 
\\

\noindent
{\bf First step.} For any  $\ybf\in\bbf(\x)$, $ B({\bf g}({\bf y}),\frac{1}{4R})\cap \mathcal{R} \subseteq \{q_0\}$, i.e. $q_0$ is the only possible $\frac{1}{4R}$ approximation of ${\bf g}({\bf y})$ with an element of $\mathcal{R}$.
\\

\noindent
{\bf Proof of the first step.}
Let $q_0'\in B({\bf g}({\bf y}),\frac{1}{4R})\cap \mathcal{R}$. Assume that $q_0\neq q_0'$. In order to get a contradiction we will use the Taylor expansion of ${\bf g}$ about $\x$ at each place, i.e.
\[q_0+g_{\nu}(y_{\nu})=q_0+g_{\nu}(x_{\nu})+\g g_{\nu}(x_{\nu})\cdot(x_{\nu}-y_{\nu})+\sum_{i.j}\bar{\Phi}_{ij}(g_{\nu})(\bullet)(x_{\nu}^{(i)}-y_{\nu}^{(i)})(x_{\nu}^{(j)}-y_{\nu}^{(j)}),\]
where the arguments of $\bar{\Phi}_{ij}(g_{\nu})$ are some of the  components of $x_{\nu}$ and $y_{\nu}$. We divide the proof into two parts:
\begin{itemize}
\item[(i)] $\widetilde{\kappa_{\mathcal{R}}}=0$. In this case, $|q_0+{\bf g}(\ybf)|_{S}<\frac{1}{4}$ because of the Taylor expansion and the following inequalities,
\begin{itemize}\item[$\bullet$] $|q_0+g_{\nu}(x_{\nu})|_{\nu}\le \frac{1}{4d_{\infty}}$ because of (P1),
\item[$\bullet$] $|\g g_{\nu}(x_{\nu})\cdot(x_{\nu}-y_{\nu})|_{\nu} \le \frac{1}{4d_{\infty}}$ because of the way we defined ${\bf B}(\x)$,
\item[$\bullet$] $|\sum_{i.j}\bar{\Phi}_{ij}(g_{\nu})(\bullet)(x_{\nu}^{(i)}-y_{\nu}^{(i)})(x_{\nu}^{(j)}-y_{\nu}^{(j)})|_{\nu} \le \frac{1}{4d_{\infty}}$ because of (C2), (P2), and the definition of ${\bf B}(\x)$.
\end{itemize}
so $|q_0-q'_0|_S<\frac{1}{2},$ which says $q_0=q'_0.$  
\item[(ii)] $\widetilde{\kappa_{\mathcal{R}}}>0$. Similar to the previous case, we will compare the maximum possible values of each of the three above expressions at each place $\nu$. The first one is always small. It is enough to compare the last two. Because of the way we defined ${\bf B}(\x)$, the second term is less than $\frac{1}{4R^{1/\widetilde{\kappa_{\mathcal{R}}}}}$ (resp. $\frac{R^{1/\kappa_{\mathcal{R}}}}{2}$) for $\nu \in {S_{\mathcal{R}}}^c$ (resp. $\nu \in S_{\mathcal{R}}$). Indeed the third term is also less than these values because of  (P2) (resp. (P3)), $|q_i|_S\le T_i$ (resp. $T_i/2 \le |q_i|_S$), and $\epsilon$ being small. So we have $\p |q'_0-q_0|_{\nu}<\frac{1}{4},$ which contradicts the product formula, since we have $q_0, q'_0\in\mathcal{R}\subseteq\bbz_S$.
\end{itemize}

\noindent
{\bf Second step.}
For any $\nu\in S$ and $\ybf\in\bbf(\x)$, we have $$\|\g g_{\nu}(y_{\nu})-\g g_{\nu}(x_{\nu})\|_{\nu}<\|\g g_{\nu}(x_{\nu})\|_{\nu}/4.$$ 
\\

\noindent
{\bf Proof of the second step.} This time, we will use the Taylor expansion of $\partial_i g_{\nu}$ about $x_{\nu}$. So let $\mathbf{z}=(z_{\nu})$ where $y_{\nu}=x_{\nu}+z_{\nu}$. In this setting, we have $$\begin{array}{ll}\partial_i
g_{\nu}(y_{\nu})&=\partial_i g_{\nu}(x_{\nu})+\sum_{j}
\bar{\Phi}_{j}(\partial_i g_{\nu})(\bullet)z_{\nu}^j\\&=\partial_i g_{\nu}(x_{\nu})+
\sum_{j}(\bar{\Phi}_{ji}(g_{\nu})(\bullet)+\bar{\Phi}_{ji}(g_{\nu})(\bullet))z_{\nu}^j,\end{array}$$
where the arguments of $\bar{\Phi}_{ij}(g_{\nu})$ and $\bar{\Phi}_{j}(\partial_i g_{\nu})$, as before, are some of the  components of $x_{\nu}$ and $y_{\nu}$.
Now similar to the first step, one can argue and get the following inequalities, which complete the proof of the second step.
\begin{itemize}

\item[(i)] If $\nu\in {S_{\mathcal{R}}}^c$ then $$|\partial_i g_{\nu}(y_{\nu})-\partial_i
g_{\nu}(x_{\nu})|_{\nu}<|z_{\nu}|_{\nu}\leq\frac{1}{4R^{\frac{1}{\widetilde{\kappa_{\mathcal{R}}}}}\|\nabla
{g}_{\nu}(x_{\nu})\|_{\nu}}\leq \frac{\|\nabla {g}_{\nu}(x_{\nu})\|_{\nu}}{4},$$

\item[(ii)] If $\nu\in S_{\mathcal{R}}$ then $$|\partial_i g_{\nu}(y_{\nu})-\partial_i
g_{\nu}(x_{\nu})|_{\nu}<2d_\infty|q|_\nu|z_{\nu}|_{\nu}\leq 2d_\infty|q|_\nu\rho\leq \frac{\|\nabla {g}_{\nu}(x_{\nu})\|_{\nu}}{4},$$ where $\rho=\frac{1}{4d_{\infty}\|\nabla
{g}_{\nu}(x_{\nu})\|_{\nu}}$ if $S_{\mathcal{R}}=S$ and $\rho=\frac{R^{\frac{1}{\kappa_{\mathcal{R}}}}}{4d_{\infty}\|\nabla
{g}_{\nu}(x_{\nu})\|_{\nu}}$ otherwise. 
\end{itemize} 
\vspace{1mm}

\noindent
{\bf Third step.} $|\pi_{\nu}(\mathcal{A}_{\q}\cap\bbf(\x))|\le C' (\delta T^{-g(\mathcal{R})})^{\frac{1}{\kappa}} r_{\nu} |\pi_{\nu}({\bf B}(\x))|,$ where $r_{\nu}= \frac{1}{R^{1/\kappa_{\mathcal{R}}}}$ (resp. $R^{1/\widetilde{\kappa_{\mathcal{R}}}}$, 1) for $\nu\in S_{\mathcal{R}}$ (resp. $\nu\in S_{\mathcal{R}}^c$, $\nu\in S_{\mathcal{R}}=S$) and $\pi_{\nu}$ is projection into the $\bbq_{\nu}^{d_{\nu}}$.
\\

\noindent
{\bf Proof of the third step.} Without loss of generality, we may and will assume that $\|\g g_{\nu}(x_{\nu})\|_{\nu}=|\d_1g_{\nu}(x_{\nu})|_{\nu}.$  In fact, we will show that the considered set is thin in the $e_1$ direction, and  it gives us the factor saving. So let $\ybf, \ybf'\in\mathcal{A}_{\q}\cap\bbf(\x),$ and assume that $\pi_{\nu}(\ybf')=\pi_{\nu}(\ybf)+\alpha e_1$. Note that by the first step and (P1), for some $q_0\in\mathcal{R}$, we have $|q_0+{\bf g}(\ybf)|^{\kappa}_S\le \delta T^{-g(\mathcal{R})}$ and  $|q_0+{\bf g}(\ybf')|^{\kappa}_S\le \delta T^{-g(\mathcal{R})}$, and so \begin{equation}\label{Snorm}|{\bf g}(\ybf')-{\bf g}(\ybf)|^{\kappa}_S \le 2^{\kappa} \delta T^{-g(\mathcal{R})}.\end{equation} As always set $y_{\nu}=\pi_{\nu}(\ybf)$ and $y'_{\nu}=\pi_{\nu}(\ybf')$. 
 
\begin{itemize}

\item[(i)] $\nu\in S\backslash \{\infty\}$. Again we use the Taylor expansion to get a ``mean value theorem" at the norm level. $$g_{\nu}(y_{\nu}+\alpha e_1)-g_{\nu}(y_{\nu})=\d_1g_{\nu}(y_{\nu})\alpha+\Phi_{11}g(\bullet)\alpha^2,$$ as before a norm comparison, gives us \begin{equation}\label{finite}|g_{\nu}(y_{\nu}+\alpha e_1)-g_{\nu}(y_{\nu})|_{\nu}=|\d_1g_{\nu}(y_{\nu})|_{\nu}|\alpha|_{\nu}\end{equation}
\item[(ii)] $\nu=\infty$. Here we have the mean value theorem and so for some $z_{\infty}$, \begin{equation}\label{infinite}g_{\infty}(y_{\infty}+\alpha e_1)-g_{\infty}(y_{\infty})=\d_1g_{\infty}(z_{\infty})\alpha\end{equation}

\end{itemize}

\noindent

Now by fixing the last $d_{\nu}-1$ entries, we slice our set, and equations~\ref{finite} and~\ref{infinite} coupled with inequality~\ref{Snorm} and the second step tell us that the measure of each slice is at most $C'' \frac{(\delta T^{-g(\mathcal{R})})^{\frac{1}{\kappa}}}{\|\g g_{\nu}(x_{\nu})\|_{\nu}}=C''(\delta T^{-g(\mathcal{R})})^{\frac{1}{\kappa}} r_{\nu}  \times \mbox{radius of } \pi_{\nu}({\bf B}(\x))$. Now direct use of Fubini's theorem completes the proof of this step.
\\

\noindent
{\bf Final step.} For any $\nu\in S$, $\{\pi_{\nu}({\bf B}(\x))\}_{\x\in\mathcal{A_{\q}}}$ is a covering of $\pi_{\nu}(\mathcal{A}_{\q})$. Using Besicovitch covering lemma (see in~\cite{KT} for details on this) and the third step, we can conclude that \[|\pi_{\nu}(\mathcal{A}_{\q})|\le C''' (\delta T^{-g(\mathcal{R})})^{\frac{1}{\kappa}} r_{\nu} |U_{\nu}|, \] for some universal constant $C'''$.  the following inequalities complete the proof: \[|\mathcal{A}_{\q}| \le \prod_{\nu\in S} |\pi_{\nu}(\mathcal{A}_{\q})| \le C \prod_{\nu\in S}  ((\delta T^{-g(\mathcal{R})})^{\frac{1}{\kappa}} r_{\nu}) |{\bf U}|=C \delta\h T^{-g(\mathcal{R})} |\U|. \]

%%%%%%%%%%%%%%%%%%%%%%%%%%%%%%%%%%%%%%%%%%%%%%%%%%%%%%%%%%%%%%%%%%%%%%%%%%%%
%%%%%%%%%%%%%%%%%%%%%%%%%% Good Functions %%%%%%%%%%%%%%%%%%%%%%%%%%%%%%%%%%
%%%%%%%%%%%%%%%%%%%%%%%%%%%%%%%%%%%%%%%%%%%%%%%%%%%%%%%%%%%%%%%%%%%%%%%%%%%%
\section{Good functions}
In this section, first we recall the notion of a {\it good
function} and a few known theorems, then we establish the
``{\it goodness}" of a family of $\nu$-adic analytic functions, which
will be needed in the proof of theorem~\ref{<}.

\begin{definition}\cite{KM} Let
$C$ and $\alpha$ be positive real numbers, a function $\fbf$
defined on an open set $\mathbf{V}$ of $X=\prod_{\nu\in S}
\bbq_{\nu}^{m_\nu}$ is called $(C, \alpha)$-good, if for any open
ball $\mathbf{B}\subset \mathbf{V}$ and any $\vare >0$ one has
$$|\{\mathbf{x}\in \mathbf{B}|\hspace{1mm}
\|\mathbf{f}(\mathbf{x})\|< \vare\cdot\sup_{\mathbf{x}\in
\mathbf{B}}\|\mathbf{f}(\mathbf{x})\|\}|\leq C
\hspace{1mm}\vare^{\alpha}|\mathbf{B}|.$$ 
\end{definition} 
\noindent
The following is tautological consequence of the above definition.
\begin{lem}
Let $X=\prod_{\nu\in S} \bbq_{\nu}^{m_\nu}, \mathbf{V}$ and $\mathbf{f}$ be as above then
\begin{itemize}
\item[(i)] $\mathbf{f}$ is $(C,\alpha)$-good on $\mathbf{V}$ if and only if $\|\mathbf{f}\|$ is $(C,\alpha)$-good.

\item[(ii)] If $\mathbf{f}$ is $(C,\alpha)$-good on $\mathbf{V}$,
then so is $\lambda\mathbf{f}$ for any $\lambda\in \bbq_S.$

\item[(iii)] Let $I$ be a countable index set, if $\mathbf{f}_i$ is $(C,\alpha)$-good on
$\mathbf{V}$ for any $i\in I$, then so is $\sup_{i\in
I}\|\mathbf{f}\|.$

\item[(iv)] If $\mathbf{f}$ is $(C,\alpha)$-good on $\mathbf{V}$
and $c_1\leq
\|\mathbf{f}(\mathbf{x})\|_S/\|\mathbf{g}(\mathbf{x})\|_S\leq
c_2$, for any $x\in \mathbf{V},$ then $\mathbf{g}$ is
$(C({c_2}/{c_1})^{\alpha},\alpha)$-good on $\mathbf{V}.$
\end{itemize}
\end{lem}
\noindent
Let us recall the following lemma from~\cite[lemma 2.4]{KT}.

\begin{lem}~\label{polygood}
Let $\hspace{1mm}\nu$ be any place of $\bbq$ and $p\in\bbq_{\nu}[x_1,\cdots,x_d]$
 be a polynomial of degree not greater than $l$. Then there exists
 $\hspace{1mm}C=C_{d, l}\hspace{1mm}$ independent of $p$, such that $p$ is $(C,1/{dl})$-good on $\bbq_{\nu}$.
\end{lem}
\noindent
Next we state a variation of \cite [theorem 3.2] {KT}
without proof.
\begin{thm}~\label{good}
Let $V_1, \cdots, V_d$ be nonempty open sets in $\bbq_{\nu}$, Let
$k\in \bbn$, $A_1, \cdots, A_d, A'_1, \cdots, A'_d$ positive real
numbers and $f\in C^k(V_1\times \cdots \times V_d)$ be such that
$$A_i\leq |\Phi_i^{k}f|_{\nu}\leq A'_i \hspace{2mm}\mbox{\rm{on}}
\bigtriangledown^{k+1}V_i\times \prod_{j\neq i} V_j,
\hspace{1mm}i=1\cdots, d .$$ Then $f$ is $(C, \alpha)$-good on
$V_1\times\cdots\times V_d$, where $C$ and $\alpha$ depend only
on $k, d, A_i$, and $A'_i$ .
\end{thm}
\noindent
Another useful fact which can be easily adapted to the $\nu$-adic
calculus is proposition 3.4 of~\cite{BKM}.
\begin{thm}~\label{compact}
Let $U$ be an open neighborhood of $x_0\in\bbq_{\nu}^m$ and let
$\mathcal{F}\subset C^l(U)$ be a family of functions $f : U
\rightarrow \bbq_{\nu}$ such that
\begin{itemize}
\item[1.] $\{\nabla f| f\in \mathcal{F}\}$ is compact in $C^{l-1}(U)$

\item[2.] $ \inf_{f\in \mathcal{F}} \sup_{|\beta| \leq l}|\partial_{\beta}f(x_0)|> 0.$
\end{itemize}
Then there exist a neighborhood $V\subseteq U$ of $x_0$ and
positive numbers $C=C(\mathcal{F})$ and
$\alpha=\alpha(\mathcal{F})$ such that for any $f\in \mathcal{F}$
\begin{itemize}
\item[(i)] $f$ is $(C, \alpha)$-good on $V$.

\item[(ii)] $\nabla f$ is $(C, \alpha)$-good on $V$.
\end{itemize}
\end{thm}

\begin{proof}
The argument in \cite [proposition 3.4]{BKM} goes through using
theorem \ref{good}.
\end{proof}

\begin{cor}
\label{linear}
Let $f_1, f_2, \cdots, f_n$ be analytic functions from a neighborhood $U$ of $x_0$ in $\bbq_{\nu}^m$ to $\bbq_{\nu}$, such that $1,\hspace{1mm}f_1, f_2, \cdots, f_n$ are linearly independent on any neighborhood of $x_0$, then
\begin{itemize}
\item[(i)] There exist a neighborhood $V$ of $x_0$, $C \hspace{1mm}\& \hspace{1mm}\alpha>0$ such that any linear combination of $\hspace{1mm}1, f_1, f_2, \cdots, f_n$ is $(C, \alpha)$-good on $V$.

\item[(ii)] There exist a neighborhood $V'$ of $x_0$, $C' \hspace{1mm}\& \hspace{1mm}\alpha'>0$ such that for any $\hspace{1mm}d_1, d_2, \cdots, d_n \in \bbq_{\nu}$, $\|\sum_{k=1}^{n} d_i \nabla f_i\|$ is $(C', \alpha')$-good.
\end{itemize}
\end{cor}

\begin{proof}
Let $\mathcal{F}=\{d+D\cdot(f_1, \cdots, f_n)|\hspace{1mm} d\in \bbq_{\nu},\hspace{1mm} D\in \bbq_{\nu}^n,\hspace{1mm} \|D\|=1\}$. By our assumptions on $f_1, \cdots, f_n$, the family  $\mathcal{F}$ satisfies the conditions of theorem \ref{compact} which gives the corollary.
\end{proof}
\noindent
The following theorem is the main result of this section. This
technical theorem is crucial in the proof of theorem~\ref{<}. Let us first recall the notion of {\it skew gradient} from \cite[section 4] {BKM}. For two $C^1$ functions $g_i:\mathbb{Q}_{\nu}^d\rightarrow\mathbb{Q}_{\nu},$ $i=1,2$ define $\widetilde\g(g_1,g_2):=g_1\g g_2-g_2\g g_1.$ This, as one expects from the definition, measures how far two functions are from being linearly dependent ref. loc. cit. for a discussion on this.

\begin{thm}~\label{tartibat}
Let $U$ be a neighborhood of $x_0\in \bbq_{\nu}^m$, $f_1, f_2,
\cdots, f_n$ be analytic functions from $U$ to $\bbq_{\nu}$, such that $1,\hspace{.5mm}f_1, f_2,\cdots, f_n$ are linearly independent on any open subset of $U.$ Let $F=(f_1, \cdots,
f_n)$ and $$\mathcal{F}= \{(D\cdot F,\hspace{1mm} D'\cdot F+a)|\hspace{1mm}
\|D\|=\|D'\|=\|D\wedge D'\|=1,\hspace{1mm}
 D, D'\in \bbq_{\nu}^n,\hspace{1mm} a\in \bbq_{\nu}\}.$$
 Then there exists a neighborhood $V\subseteq U$ of $x_0$ such that
\begin{itemize}
\item[(i)] For any neighborhood $B\subseteq V$ of $x_0$,
 there exists $\rho=\rho(\mathcal{F}, B)$ such t
 $\sup_{x\in B} \parallel \widetilde{\nabla} g(x)\parallel \geq \rho$ for any $g\in \mathcal{F}$.

\item[(ii)] There exist $C, \hspace{1mm}\alpha$ positive
 numbers such that $\| \widetilde{\nabla} g\|$ is $(C, \alpha)$-good on $V$, for any $g\in \mathcal{F}.$
\end{itemize}
\end{thm}

\begin{proof}
The case $\nu=\infty$ is proposition 4.1 of \cite{BKM}, so we may assume $\nu$ is a finite place. We start with part (i) proceeding by contradiction. If not, one can
find a neighborhood $B$ of $x_0$ such that for any $n$, there
would exist $g_n\in \mathcal{F}$ with $\| \widetilde{\nabla}
g_n(x)\|\leq 1/n$ for any $x\in B$. Let
$g_n=(D_n\cdot F,\hspace{1mm}D'_n\cdot F+a_n)$. If there is a bounded
subsequence of ${a_n}$, going to a subsequence, we may assume
${g_n}$ is converging to $g\in \mathcal{F}$. Therefore $\|
\widetilde{\nabla} g(x)\|=0$ for any $x\in B$ which contradicts
linearly independence of $1, f_1, \cdots, f_n$. Thus we may assume
that, $a_n\rightarrow \infty$. However $\inf_{\| D\|=1} \sup_{x\in
B} \| D\nabla F(x)\|=\delta >0$, therefore $\sup_{x\in B} \|
\widetilde{\nabla} g_n(x)\|\rightarrow \infty$ in contrary to our
assumption.
\\

\noindent
Now we prove part (ii). The proof will be divided into two parts.
First we shall deal with the``compact" part of $\mathcal{F}$, i.e.
when we have an upper bounded on $|a|$, later we will prove the
unbounded part.

\begin{lem}~\label{bounded}
Let $U \hspace{1mm}\& \hspace{1mm}F$ be as in the theorem
\ref{tartibat} and $\mathcal{F}_M$ be $$\{(D\cdot F,
D'\cdot F+a)|\hspace{1mm} \|D\|=\|D'\|=\|D\wedge D'\|=1,\hspace{1mm} D,
D'\in \bbq_{\nu}^n,\hspace{1mm} a\in \bbq_{\nu},\hspace{1mm}
|a|_{\nu}\leq M\}.$$ Then there exist a neighborhood $V=V_M$ of
$x_0$, $\hspace{1mm}C=C_M$ and $\alpha=\alpha_M >0$ such that $\|
\widetilde{\nabla} g\|$ is $(C, \alpha)$-good on $V$ for any $g\in
\mathcal{F}_M$.
\end{lem}
\begin{proof}
Replacing $F(x)$ by $F(x+x_0)$ we may assume that $x_0=0.$ Then rescaling $x$ by $rx$ we may and will assume all the Taylor coefficients of $f_i$'s are in $\mathbb{Z}_{\nu},$ and $U\subseteq p_{\nu}\mathbb{Z}_{\nu}.$ Now for any $g\in \mathcal{F}_M$ there is $h_g\in \rm{GL}_m(\bbz_{\nu})$
such that all of the components of $\widetilde{\nabla}(g\circ h_g)$ are
non-zero functions. By the compactness assumption on
$\mathcal{F}_M$, we may find $h_1, \cdots, h_k\in
\rm{GL}_m(\bbz_{\nu})$ such that$$\sup_{\substack{1\leqslant
j\leqslant k\\\hspace{.5mm}x\in U}} |(\widetilde{\nabla}(g\circ
h_j))_i(x)|\geq \delta\hspace{2mm}\mbox{\rm{for}}\hspace{2mm} 1\leq
i\leq m.$$ Hence we can find $b=b_{\delta}$ with the following property: for any
$g\in\mathcal{F}_M$ there exists $1\leq j\leq k,$ such that for any $1\leq i\leq m$ one can find a multi-index $\beta$ with
$|\beta|\leq b,$ and $|\partial_{\beta}(\widetilde\g(g\circ h_j))_i(0)|\geq
\delta$. Using theorem \ref{compact}, there exist a neighborhood
$V'$ of the origin, $C$ and $\alpha >0$ such that for any $g\in
\mathcal{F}_M$ one can find $1\leq j\leq k$ so that
$\para\widetilde{\nabla}(g\circ h_j)\para$ is $(C,
\alpha)$-good on $V'$, which says $\|\widetilde{\nabla}g\|$ is good on
some $V'$ for any $g\in \mathcal{F}_M$.
\end{proof}
\noindent
To prove the unbounded part, we need the following lemma.

\begin{lem}~\label{gamma}
Set $H=(p_1, p_2, \cdots, p_n)$ where $p_i\in \bbz_{\nu}[x_1,
\cdots, x_m]$ are linearly
 independent polynomials of degree $\leq\hspace{1mm} l$. For any positive real number $r$
 let $H_r(x)=\frac{H(\lfloor r\rfloor_{\nu}x)}{\lfloor r\rfloor_{\nu}^{l}}$.
 Then there exist $\gamma$ and $0<s<1$ such that for any $D, D' \in \bbq_{\nu}^n$
 with $\|D\|=\|D'\|=\|D\wedge D'\|=1$, any $a\in\bbq_{\nu}$ with $|a|\geq p_{\nu}^l$
  and $r<s$ one has $$\|\widetilde{\nabla}P_r(x)\|_{{B}_1}\geq\hspace{1mm}\gamma(1+\|P_r\|_{{B}_1}),$$
 where $P_r=(D\cdot H_r,\hspace{1mm} D'\cdot H_r+\frac{a}{\lfloor r\rfloor_{\nu}^l})$. 
\end{lem}
\begin{proof}
\textbf{\textit{First Claim.}} For any $p(x)=\sum_{i=0}^d c_ix^i\in \bbq_{\nu}[x]$ and $0<\delta<1$, there exists $s$, such that for any $r<s$, one has
$$\sup_{x\in {B}_r}|p(x)|\hspace{1mm}\geq\hspace{1mm}\frac{|c_k|r^k}{k^k} , \hspace{1mm} \mbox{\rm{where}} \hspace{2mm} |c_k| \geq \delta \max_i\{|c_i|\}.$$ With the understanding that $0^0=1.$ 

\noindent
\textit{Proof of the first claim.}
We will see $s=\delta$ works. If $k=0$ then $|p(x)|=|c_0|$ for $|x|<\delta$ and there is nothing to prove. Otherwise, there exists $k>0$ such that $|\Phi_{k}(p)(x_1,\cdots,x_{k+1})|=|c_k|$ for $|x_i|<\delta.$ 
Take $x_1,\cdots,x_{k+1}$ such that $|x_i-x_j|\geq r/k$ for $1\leq i\neq j\leq k+1,$ where $r<s=\delta,$
Let $$q(x)=\sum_{i=1}^{k+1}p(x_i)\frac{\prod_{j\neq i}(x-x_j)}{\prod_{j\neq i}(x_i-x_j)},$$
 be the degree $k$, Lagrange polynomial of $p$ with respect to  $x_1,\cdots,x_{k+1}.$
 Then we get
 $$|c_k|=|\Phi_k(p)(x_1,\cdots,x_{k+1})|=|\Phi_k(q)(x_1,\cdots,x_{k+1})|$$
 $$=|\sum_{i=1}^{k+1}\frac{p(x_i)}{\prod_{j\neq i}(x_i-x_j)}|\leq\frac{\|p\|_{\mathbf{B}_r}}{(r/k)^k} .$$
  Therefore $\|p\|_{\mathbf{B}_r}\geq |c_k|r^k/{k^k}$ as we wanted to show.\vspace{1mm}

\noindent
\textbf{\textit{Second claim.}}
There exist $a_0, C',s>0$ such that for any $Q$ in $$ \widetilde{\mathcal{G}}_{a_0}=\{\widetilde{\nabla}(D\cdot H,D'\cdot H+a)|D,D'\hspace{1mm}\mbox{orthonormal vectors,}\hspace{1mm}|a|\geq a_0\},$$ and $r<s$, one has 
\begin{equation}~\label{norm}\para Q\para_{B_r}\ge C'|a|r^{l-1}.\end{equation}
\textit{Proof of the second claim.} 
As the family $\mathcal{G}=\{D\cdot H|\hspace{1mm}\|D\|=1\}$ is a compact family of functions, and $p_i$'s are linearly independent, there exists $\delta'>0$ such that $\para D\cdot H\para\ge \delta'$, for any $\para D\para=1$. Thus for any polynomial $p\in\mathcal{G}$, there exists a multi-index $\beta$ with $|\beta|=k\le l$ such that $\Phi_{\beta}(p)(\mathbf{0})\ge \delta'$. Hence for any such $p$, one may find $h_p\in \GL_m(\bbz_{\nu}), \delta_p,$ and $s_p,$ such that $|\Phi_1^k(p\circ h_p)(x_0,\cdots,x_k)|\ge \delta_p$ for any $x_j$'s with norm at most $s_p$. Now by the compactness of $\GL_m(\bbz_{\nu})$ and $\mathcal{G}$, there are $h_1,\cdots,h_t\in \GL_m(\bbz_{\nu})$, and positive numbers $\delta'', s'$ such that for any $p\in\mathcal{G}$, $|\Phi_1^k(p\circ h_i)(x_0,\cdots,x_k)|\ge \delta''\hspace{3mm}(*)$ for some $1\le i\le t$, and any $x_j$ with norm at most $s'$.
\\

\noindent
Now let $D$ and $D'$ be two orthonormal vectors and $g(x)=\widetilde{\g}(D\cdot H,D'\cdot H+a).$ As $\para g\para_{B_r}=\para g\circ h\para_{B_r}$ for any $h\in \GL_m(\bbz_{\nu})$, we may and will replace $g$ by $g\circ h_i$, where $i$ has been chosen such that $(*)$ holds for $p=D\cdot H$. Hence if $a_0=p_{\nu}^l,$ the coefficient of $x_1^{k-1}$ in the first component of $\widetilde{\g}(D\cdot H\circ h_i,D'\cdot H\circ h_i+a)$ has norm at least $\delta''|a|/p_{\nu}^l$, and moreover all the coefficients have norm at most $|a|$. Now let $x_2=\cdots=x_n=0$. We would get a one-variable polynomial whose coefficient of term $x_1^{k-1}$ has norm at least $\delta''/p_{\nu}^l$ times the maximum norm of all the coefficients.
Thus the first claim completes the proof of the second claim.

\noindent
\textbf{\textit{Final step.}}
let $P(x)=(D\cdot H,\hspace{1mm} D'\cdot H+a)$. Note that $\|P_{r}\|_{{B}_1}=|a|/r^l$, $\|\widetilde{\nabla}P_r\|=\|\widetilde{\nabla}P\hspace{1mm}\|/{r^{2l-1}}$. Using these and~\ref{norm}, one sees that $\gamma =\frac{C'}{2}$ works.
\end{proof}
Before proving the unbounded part, let us recall and give the needed modification of lemma 3.7 of \cite{BKM}. 
\begin{lem}\label{del-good}
Let $B\subseteq \bbq_{\nu}^d$ be an open ball of radius $r$, and $\widetilde{B}$ be the ball with the same center as $B$ and of radius $(p_{\nu}+1)\cdot r$. Let $f$ be a continuous function on $\widetilde{B}$. Suppose $C,\alpha$, and $\delta$ are positive real numbers such that
\[
|\{x\in B'|\h|f(x)|<\vare\cdot\sup_{x \in B'} |f(x)|\}|\le C \vare^{\alpha} |B'|,
\]
for any ball $B'\subseteq \widetilde{B}$ and $\vare\ge \delta$. Then $f$ is $(C,\alpha')$-good on $B$ whenever $0<\alpha'<\alpha$ and $Cp_{\nu}\delta^{\alpha-\alpha'}\le 1$.
\end{lem}
\begin{proof} The same argument as in ~\cite{BKM} works in the non-Archimedean setting, too. However we have to replace $\sup_{x\in B(y)}|f(x)|=\vare\cdot\sup_{x\in B}|f(x)|$, with
 \[\sup_{x\in B(y)}|f(x)|\le\vare\cdot\sup_{x\in B}|f(x)|\le\sup_{x\in B'(y)}|f(x)|,
 \]
 where $B'(y)$ is a ball  centered at $y$ whose radius is $p_{\nu}$ times the radius of $B(y)$.
Then use the covering of $B'(y)$'s instead of $B(y)$'s and note that  \[|B'(y)|=p_{\nu} |B(y)|.\]
\end{proof}
\begin{lem}~\label{unbounded}
Let $x_0,\hspace{1mm} U\hspace{1mm}\&\hspace{1mm}F$ be as in
theorem \ref{tartibat}, and 
$$\mathcal{F}'_M=\{(D\cdot F,\hspace{1mm}D'\cdot F+a)|\hspace{1mm} D, D'\hspace{1mm}\mbox{\rm{orthonormal},}\hspace{1mm}
a\in\bbq_{\nu},\hspace{1mm} |a|\geq M \}.$$  Then for sufficiently
large $M$ there exist neighborhood $V$ of $x_0$ and positive
numbers $C\hspace{1mm}\&\hspace{1mm} \alpha,$ such that for any
$g\in\mathcal{F}'_M$, $\|\widetilde{\nabla}g\|$ is
$(C,\alpha)$-good function on $V$ for any $g\in \mathcal{F}'_M$.
\end{lem}
\begin{proof}
Without loss of generality we assume $x_0=0$ and
$f_i(x)=\sum_{\beta\in \bbz^m} a_{\beta}^{(i)}x^{\beta},$
$\hspace{1mm}1\leq i\leq n$, where $a_{\beta}^{(i)}\in \bbz_{\nu}$
and $\|x\|<1$. Let $p_{l}^{(i)}$ be the $l^{th}$ degree Taylor
polynomial of $f_i$ then $|f_i(x)-p_{l}^{(i)}(x)|\leq
\|x\|^{l+1}$. Let $l$ be large enough such that $1, p_{l}^{(1)},
\cdots, p_{l}^{(n)}$ are linearly independent also let $r_0<s$ be
small enough such that 
$$2p_{\nu}\hspace{1mm}C_{m,2l-2}
(\frac{8r_0}{\gamma})^{1/m(2l-1)(2l-2)}\leq 1,$$ 
where
$s,\hspace{1mm} \gamma$ are given as in lemma \ref{gamma} and $C_{m,2l-2}$ is as in lemma \ref{polygood}. Now
take $M\geq p_{\nu}^l$ and consider $g(x)=(C\cdot F(x),\hspace{1mm}
D\cdot F(x)+a)$ from $\mathcal{F}'_M$ furthermore set
$p(x)=(C\cdot(p_{l}^{(1)}, \cdots, p_{l}^{(n)}),$ $\hspace{1mm}
D\cdot \hspace{1mm}(p_{l}^{(1)}, \cdots, p_{l}^{(n)})+a)$. By lemma~\ref{del-good}, it is enough to prove the following:

\noindent
$$(*)\begin{array}{l}\mbox{For} \h\frac{8r_0}{\gamma}\leq\vare\leq 1,\h\mbox{any ball}\h
B=B_r(x_1)\subseteq {B}_{r_0}(0)\h\mbox{and any}\h g\in
\mathcal{F}'_M,\h\mbox{one has:}\\
|\{ x\in B_r(x_1)|\hspace{1mm}\|\widetilde{\nabla}g(x)\|< \vare\cdot\sup_{x\in {B}}\|\widetilde{\nabla}g(x)\|\}|\leq 2C_{m,2l-2}\hspace{1mm}\vare^{1/m(2l-2)}|{B}|.\end{array}$$ 
\vspace{1mm}
Let
$g_r(x)=\frac{g(\lfloor r\rfloor_{\nu}x+x_1)}{\lfloor
r\rfloor_{\nu}^{l}}\hspace{2mm}\&\hspace{2mm}
p_r(x)=\frac{p(\lfloor r\rfloor_{\nu}x+x_1)}{\lfloor
r\rfloor_{\nu}^{l}}$, it is clear that $(*)$ holds if and only if
$$|\{ x\in {B}_1(0)|\hspace{1mm}\|\widetilde{\nabla}g_r(x)\|<
\vare\cdot\sup_{x\in {B}_1}\|\widetilde{\nabla}g_r(x)\|\}|\leq
2C_{m,2l-2}\hspace{1mm}\vare^{1/m(2l-2)}|{B}_1|,\hspace{2mm}(\dagger)$$
where $B_1$ is the ball of radius 1 about the origin.
However for any $x\in {B}_1$, $\|g_r(x)-p_r(x)\|<r \hspace{2mm},\hspace{2mm} \|\nabla
g_r(x)-\nabla p_r(x)\|<r$. Therefore
$$\hspace{2mm}\|\widetilde{\nabla} g_r(x)-\widetilde{\nabla} p_r(x)\|\leq r(r+2)(1+\|p_r(x)\|)$$
$$\leq 3r(1+\|p_r(x)\|)\leq \frac{3r}{\gamma}\sup_{x\in {B}_1}\|\widetilde{\nabla}p_r(x)\|.$$
Hence $\hspace{2mm}\{ x\in
{B}_1(0)|\hspace{1mm}\|\widetilde{\nabla}g_r(x)\|< \vare
\hspace{1mm}\sup_{x\in {B}_1}\|\widetilde{\nabla}g_r(x)\|\}$ is a subset of
$$\{ x\in {B}_1(0)|\hspace{1mm}\|\widetilde{\nabla}p_r(x)\|
-\frac{3r}{\gamma}\sup_{x\in {B}_1}\|\widetilde{\nabla}p_r(x)\|<
\vare(1+\frac{3r}{\gamma})\sup_{x\in
{B}_1}\|\widetilde{\nabla}p_r(x)\|\}$$ $$=\{ x\in
{B}_1(0)|\hspace{1mm}\|\widetilde{\nabla}p_r(x)\|<
(\vare(1+\frac{3r}{\gamma})+\frac{3r}{\gamma})\sup_{x\in
{B}_1}\|\widetilde{\nabla}p_r(x)\|\}$$
$$\subseteq\hspace{1mm}\{x\in {B}_1|\hspace{1mm}
\|\widetilde{\nabla}p_r(x)\|<\hspace{1mm}2\vare \hspace{1mm}
\sup_{x\in {B}_1}\|\widetilde{\nabla}p_r(x)\|\}.$$ Since
each of the components of $\widetilde{\nabla}p_r(x)$ is a
polynomial of degree at most $\hspace{1mm} 2l-2$, and
$\widetilde{\g}p_r$ is not zero, $(\dagger)$ holds, which finishes
the proof.
\end{proof}

\noindent Lemmas~\ref{bounded} and ~\ref{unbounded} complete the
proof of part(ii) of theorem~\ref{tartibat}.
\end{proof}

%%%%%%%%%%%%%%%%%%%%%%%%%%%%%%%%%%%%%%%%%%%%%%%%%%%%%%%%%%%%%%%%%%%%%%%%%%%%
%%%%%%%%%%%%%%%%%%%%%% Proof of the < theorem %%%%%%%%%%%%%%%%%%%%%%%%%%%%%%
%%%%%%%%%%%%%%%%%%%%%%%%%%%%%%%%%%%%%%%%%%%%%%%%
%%%%%%%%%%%%%%%%%%%%%%%%%%%%%%%%%%%%%%%%%%%%%%%%%%%%%%%%%%%%%%%%%%%%%%%%%%%%

\section{Theorem~\ref{<} and lattices}
In this chapter, following~\cite{KM},~\cite{BKM}, and~\cite{KT} we are going
to convert the problem into a quantitative question about ``special"
unipotent flows on the space of discrete $\bbz_S-$submodules.
\noindent
In the remaining part of this article, we let
$m_{\nu}=n+d_{\nu}+1$, and we are going to work with discrete $\bbz_S-$submodules of the
$\bbq_S$-module $X=\p \bbq_{\nu}^{m_{\nu}}$. We shall denote the
standard basis of the $\nu$-factor $\bbq_{\nu}^{m_{\nu}}$ of $X$ by
$\{e_{\nu}^{0},e_{\nu}^{*1},\cdots,e_{\nu}^{*d_{\nu}},e_{\nu}^1,\cdots,e_{\nu}^n\},$
let
$W_{\nu}^*=\{e_{\nu}^{*1},\cdots,e_{\nu}^{*d_{\nu}}\}_{\bbq_{\nu}}$,
$W_{\nu}=\{e_{\nu}^1,\cdots,e_{\nu}^{n-1}\}_{\bbq_{\nu}}$, and
$\Lambda$ be the $\bbz_S$-module generated by
$\e_0,\cdots,\e_n$,
 where $\e_i=(e_{\nu}^i)_{\nu\in S}$ for any $0\le i\le n.$ Take $\delta$,
 $K_{\nu}$'s, $T_i$'s, and the function $\mathbf{f}$ as in theorem~\ref{<}, and let
$$\u_{\x}=\left(\left(\begin{array}{ccc}1&0&f_{\nu}(x_{\nu})\\0&I_{d_{\nu}}&\g f_{\nu}(x_{\nu})
\\0&0&I_n\end{array}\right)\right)_{\nu\in S}.$$ One has
$$\u_{\x}\left(\left(\begin{array}{c}p\\0\\\vec{q}\end{array}\right)\right)_{\nu\in S}
=\left(\left(\begin{array}{c}p+f_{\nu}(x_{\nu})\cdot \vec{q}\\\g f_{\nu}(x_{\nu}) \vec{q}\\
\vec{q}\end{array}\right)\right)_{\nu\in S}.$$ So if $\lambda=\left(\left(\begin{array}{c}p
\\0\\\vec{q}\end{array}\right)\right)_{\nu\in S}$ has been chosen such that $\vec{q}$ satisfies
the conditions on the set~\ref{<}, and $|(p+f_{\nu}(x_{\nu})\cdot \vec{q})_{\nu\in S}|=
\l (f_{\nu}(x_{\nu})\cdot \vec{q})_{\nu\in  S}\r$, we get an upper bound on each
of the coordinates of $\u_{\x}\lambda$. Now we shall rescale the space to put
$\u_{\x}\lambda$ into a ``small" cube by multiplying it with a diagonal element
 $\mathbf{D}=(D_{\nu}=\diag((a_{\nu}^{(0)})^{-1},(a_{\nu}^*)^{-1},
 \cdots,(a_{\nu}^*)^{-1},(a_{\nu}^{(1)})^{-1},\cdots,(a_{\nu}^{(n)})^{-1})_{\nu\in S}$,
 where $a_{\nu}^{(0)}=\lceil\delta/\vare\rceil_{\nu}, a_{\nu}^{*}=\lceil K_{\nu}/\vare\rceil_{\nu},$
 and $a_{\nu}^{(i)}=\begin{cases}\lceil T_i/\vare\rceil_{\nu}\hspace{3mm}\nu\in S_{\mathcal{R}}\\\lceil1/\vare\rceil_{\nu}\hspace{4mm}\nu\in {S_{\mathcal{R}}}^c\end{cases}$ for any $1\le i\le n$.
Having this setting in mind, we state the next theorem which proves theorem~\ref{<}.

\begin{thm}~\label{unipotent}
Let $\mathbf{U}$ and $\mathbf{f}$ be as in theorem~\ref{<}; then
for any $\x=(x_{\nu})_{\nu\in S}$, there exists a neighborhood
$\mathbf{V}=\p V_{\nu}\subseteq \mathbf{U}$ of $\x$, and a
positive number $\alpha$ with the following property: for any
$B\subseteq V$ there exists $E>0$ such that for any
$\mathbf{D}=(\diag((a_{\nu}^{(0)})^{-1},(a_{\nu}^*)^{-1},
\cdots,(a_{\nu}^*)^{-1},(a_{\nu}^{(1)})^{-1},\cdots,(a_{\nu}^{(n)})^{-1}))_{\nu\in
S}$ with 
\begin{itemize}
\item[(i)] $0<|a_{\nu}^{(0)}|_{\nu}\le 1\le
|a_{\nu}^{(1)}|_{\nu}\le\cdots\le |a_{\nu}^{(n)}|_{\nu},$ and
\item[(ii)]$0<\p|a_{\nu}^*|_{\nu}\le \p\frac{1}{|a_{\nu}^{(0)}a_{\nu}^{(1)} \cdots
a_{\nu}^{(n-1)}|_{\nu}},$ 
\end{itemize}
and for any positive number $\vare$,
one has
$$|\{\mathbf{y}\in B|\hspace{1mm}c(\mathbf{D}\u_{\mathbf{y}}\lambda)<
\vare\hspace{1mm}\mbox{\rm{for some}}\hspace{1mm}\lambda\in\Lambda\setminus\{0\}\}|
\le E\hspace{1mm}\vare^{\alpha}|B|.$$
\end{thm}
\begin{proof}[Proof of theorem~\ref{<} modulo theorem~\ref{unipotent}]
 Using a permutation without loss of generality, one can assume that
 $T_1\le T_2\le\cdots\le T_n$. Now let $\vare$ as in theorem~\ref{<}.
 It is easy to verify that if one defines $a_{\nu}^{(i)}$'s and $a_{\nu}^*$ as in the setting
 of beginning of this section, they satisfy conditions of theorem~\ref{unipotent}.
 Hence theorem~\ref{unipotent} provides us with a neighborhood $\mathbf{V}$ and a positive number $\alpha$. Using the discussion in the beginning of
 this section and the fact that $c(\x)\leq \|\x\|_S^{\kappa}$, one sees that ${\alpha}/{\kappa}$ and $\mathbf{V}$ satisfy the conditions of theorem~\ref{<}.
\end{proof}

%%%%%%%%%%%%%%%%%%%%%%%%%%%%%%%%%%%%%%%%%%%%%%%%%%%%%%%%%%%%%%%%%%%%%%%%%%%%
%%%%%%%%%%%%%%%%%%%%%% Proof of theorem{unipotent} %%%%%%%%%%%%%%%%%%%%%%%%%
%%%%%%%%%%%%%%%%%%%%%%%%%%%%%%%%%%%%%%%%%%%%%%%%%%%%%%%%%%%%%%%%%%%%%%%%%%%%
\section{Proof of theorem~\ref{unipotent}}~\label{unip}
In this section, using the following generalization of~\cite[section 4]{KM} 
proved in~\cite[section 5]{KT}, we will prove
theorem~\ref{unipotent}. Before stating the theorem, let us recall the notion of {\it norm-like map} (see~\cite[section 6]{KT}).
\begin{definition}~\label{norm-like}
Let $\Omega$ be the set of all discrete $\bbz_S$-submodules of $\p\bbq_{\nu}^{m_{\nu}}$. A function $\theta$ from $\Omega$ to the positive real numbers is called a {\it norm-like map} if the following three properties hold:
\begin{itemize}
\item[i)] For any $\Delta, \Delta'$ with $\Delta'\subseteq\Delta$ and the same $\bbz_S$-rank, one has $\theta(\Delta)\leq\theta(\Delta')$.
\item[ii)] For any $\Delta$ and $\gamma\not\in\Delta_{\bbq_S}$, one has $\theta(\Delta+\bbz_S\gamma)\leq\theta(\Delta)\theta(\bbz_S\gamma)$.
\item[iii)] For any $\Delta$, the function $g\mapsto\theta(g\Delta)$ is a continuous function of\\ $g\in\GL(\p\bbq_{\nu}^{m_{\nu}})$. 
\end{itemize}
\end{definition}

\begin{thm}~\label{poset}
Let $\mathbf{B}=\mathbf{B}(\mathbf{x}_0,r_0)\subset\prod_{\nu\in S} \bbq_{\nu}^{d_\nu}$ and $\widehat{\mathbf{B}}=\mathbf{B}(\mathbf{x}_0,3^m r_0)$ for $m=\min_{\nu}{(m_{\nu})}.$ Assume that $\mathbf{H}:\widehat{\mathbf{B}}\rightarrow \rm{GL}(\prod_{\nu\in S} \bbq_{\nu}^{m_\nu})$ is a continuous map. Also let $\mathbf{\theta}$ be a norm-like map defined on the set $\Omega$ of discrete $\bbz_S$-submodules of $\prod_{\nu\in S} \bbq_{\nu}^{m_\nu},$ and $\mathfrak{P}$ be a subposet of $\Omega$. For any $\Gamma\in\mathfrak{P}$ denote by $\psi_{\Gamma}$ the function $\mathbf{x}\mapsto \mathbf{\theta}(\mathbf{H}(\mathbf{x})\Gamma)$ on $\widehat{\mathbf{B}}.$ Now suppose for some $C,\alpha>0$ and $\rho>0$ one has
\begin{itemize}
\item[(i)] for every $\Gamma\in\mathfrak{P},$ the function $\psi_{\Gamma}$ is $(C,\alpha)$-good on $\widehat{\mathbf{B}};$

\item[(ii)] for every $\Gamma\in\mathfrak{P},\hspace{1mm}\sup_{\mathbf{x}\in \mathbf{B}}\|\psi_{\Gamma}(\mathbf{x})\|_S\geq\rho;$

\item[(iii)] for every $\mathbf{x} \in \widehat{\mathbf{B}},\hspace{2mm} \#\{\Gamma\in\mathfrak{P} |\hspace{1mm}\|\psi_{\Gamma}(\mathbf{x})\|_S\leq\rho\}<\infty.$
\end{itemize}
Then for any positive $\vare\leq\rho$ one has $$|\{\mathbf{x}\in
\mathbf{B}|\hspace{1mm}\mathbf{\theta}(\mathbf{H}\mathbf{(x)}\lambda)<\vare
\hspace{1mm}\mbox{\rm{for
some}}\hspace{1mm}\lambda\in\Lambda\smallsetminus\{0\}\}|\leq
mC(N_{((d_{\nu}),S)}D^2)^m{(\frac{\vare}{\rho})}^{\alpha}|\mathbf{B}|,$$
where $D$ may be taken to be $\hspace{1mm}
3^{d_{\infty}}\prod_{\nu\in S_f}(3p_{\nu})^{d_{\nu}},$ and $N_{((d_{\nu}),S)}$ is the Besicovich constant for the space $\p \mathbb{Q}_{\nu}^{d_{\nu}}.$

\end{thm}

\noindent
To this end, we need to define a poset
$\pfr$, a norm-liked map $\theta$, a family $\mathcal{H}$ of functions, and verify the conditions of theorem~\ref{poset}
for our choices of $\mathfrak{P}$, $\theta$, and any function $\bf{H}$ in $\mathcal{H}$.
We shall start with introducing a norm-like map
$\theta$ from $\p\bigwedge \bbq_{\nu}^{m_{\nu}}$ to $\bbr^{+}$, and
then ``restrict" it to the poset of discrete $\bbz_S$-submodules of
$\p\bbq_{\nu}^{m_{\nu}}$. For each $\nu\in S$ let $\mathcal{I}^{*}_{\nu}$ be the ideal
generated by $e_{\nu}^{*i}\wedge e_{\nu}^{*j}$, for $1\le i,j\le
d_{\nu}$, and $\pi_{\nu}$ be the natural map from
$\bigwedge\bbq_{\nu}^{m_{\nu}}$ to
$\bigwedge\bbq_{\nu}^{m_{\nu}}/\mathcal{I}^{*}_{\nu}$. Define
$\theta_{\nu}(x_{\nu})=\para
\pi_{\nu}(x_{\nu})\para_{\pi_{\nu}(\bfr_{\nu})}$, where $\bfr_{\nu}$ is
the standard basis of $\bigwedge\bbq_{\nu}^{m_{\nu}}$, and let
$\theta(\x)=\p\theta_{\nu}(x_{\nu})$. For any
discrete $\bbz_S$-submodule $\Delta$ of $\p\bbq_{\nu}^{m_{\nu}}$,
let $\theta(\Delta)=\theta(\x^{(1)}\wedge\cdots\wedge\x^{(r)})$,
where $\{\x^{(1)},\cdots,\x^{(r)}\}$ is a $\bbz_S$-base of $\Delta$.
Using the product formula, it is easy to see that $\theta(\Delta)$ is well-defined, and it is a norm-like map. Now
let $\pfr$ be the poset of primitive $\bbz_S$-submodules of
$\Lambda,$ where $\Lambda$ is defined in setion 4. Let $\mathcal{H}$ be the family of functions
$$\bf{H}:\mathbf{U}=\p U_{\nu}\rightarrow {\rm
GL}(\p\bbq_{\nu}^{m_{\nu}})\hspace{3mm}\mbox{\rm{where}}\hspace{3mm}
\bf{H}(\x)=\mathbf{D}\u_{\x},$$ for any $\mathbf{D}$ satisfying
conditions of theorem~\ref{unipotent}. Since the restriction of
$\theta$ to $\p\bbq_{\nu}^{m_{\nu}}$ is the same as the function
$c$, to prove theorem~\ref{unipotent}, it suffices
to find a neighborhood $\mathbf{V}$ of $\x$ and establish the
following statements for such $\mathbf{V}$.
\begin{itemize}
\item[(I)] There exist $C,\alpha>0$, such that all the functions
$\mathbf{y}\mapsto\theta(\mathbf{H} (\ybf)\Delta)$, where $\bf{H}\in\mathcal{H}$
 and $\Delta\in\pfr$ are $(C,\alpha)$-good on $\mathbf{V}$.
\item[(II)]For all $\ybf\in\mathbf{V}$ and $\bf{H}\in\mathcal{H}$, one
has
$\#\{\Delta\in\pfr|\hspace{1mm}\theta(\mathbf{H} (\ybf)\Delta)\le1\}<\infty.$
\item[(III)] For every ball $\mathbf{B}\subseteq\mathbf{V}$,
 there exists $\rho>0$ such that $\sup_{\ybf\in\bbf}\theta(\mathbf{H}(\ybf)\Delta)\ge\rho$
 for all $\bf{H}\in\mathcal{H}$ and $\Delta\in\pfr$.
\end{itemize}
\noindent
We now verify (I-III) which will finish the proof of the theorem~\ref{<}\vspace{1mm}

\noindent \textit{\textbf{Proof of (I).}} Let ${\rm
rank}_{\bbz_S}\Delta=k\le n+1$, and let
$(\dbf\Delta)_{\nu}$ be the $\mathbb{Q}_{\nu}$-span of the projection of $\dbf\Delta$ to the $\nu$ place; then by
lemma~\ref{base} $\dim_{\bbq_{\nu}}(\dbf\Delta)_{\nu}=k$, for any
$\nu\in S$. We choose an orthonormal set
$x_{\nu}^{(1)},\cdots,x_{\nu}^{(k-1)}\in (\dbf\Delta)_{\nu}\cap
W_{\nu}\oplus\bbq_{\nu}e_{\nu}^{n}.$ By adding $e_{\nu}^0$ and
possibly another vector $x_{\nu}^{(0)}$ from
$(\dbf\Delta)_{\nu}\oplus\bbq_{\nu}e_{\nu}^0$ to the set of
$x_{\nu}^{(i)}$'s, we can get an orthonormal base of
$(\dbf\Delta)_{\nu}\oplus\bbq_{\nu}e_{\nu}^0$. 
Let $\{\ybf^{(1)},\cdots,\ybf^{(k)}\}$ be a $\bbz_S$-base of $\Delta$. 
Therefore
$\theta(\dbf\Delta)=\theta(\dbf\mathcal{Y})$, where
$\mathcal{Y}=\ybf^{(1)}\wedge\cdots\wedge\ybf^{(k)}$. Take
$a_{\nu},b_{\nu}\in\bbq_{\nu}$ such that
$$(\dbf\mathcal{Y})_{\nu}=a_{\nu} e_{\nu}^0\wedge
x_{\nu}^{(1)}\wedge\cdots\wedge
x_{\nu}^{(k-1)}+b_{\nu}x_{\nu}^{(0)}\wedge\cdots\wedge
x_{\nu}^{(k-1)}.$$ Let $\g^*\bar
g(x_{\nu})=\sum_{i=1}^{d_{\nu}}\d_i \bar g(x_{\nu})e_{\nu}^{*i},$
for any function $\bar g$ from an open subset of
$\bbq_{\nu}^{d_{\nu}}$ to $\bbq_{\nu}$, and define
$\widetilde{\g}^*(g)(x)=g_1(x)\g^*g_2(x)-g_2(x)\g^*g_1(x)$, where
$g_1$ and $g_2$ are two functions from an open subset of
$\bbq_{\nu}^{d_{\nu}}$ to $\bbq_{\nu}$, and
$g(x)=(g_1(x),g_2(x))$. Let us also define
$\hat{\mathbf{f}}(\x)=(\hat{f}_{\nu}(x_{\nu}))_{\nu\in S},$ where
$$\hat{f}_{\nu}(x_{\nu})=(1,0_{d_{\nu}},
\frac{a_{\nu}^{(1)}}{a_{\nu}^{(0)}}f_{\nu}^{(1)}(x_{\nu}),\cdots,
\frac{a_{\nu}^{(n)}}{a_{\nu}^{(0)}}f_{\nu}^{(n)}(x_{\nu})).$$ In this
setting it is easy to see that
$$(\dbf\u_{\x}\dbf^{-1})_{\nu}w=w+(\hat{f}_{\nu}(x_{\nu})\cdot
w)e_{\nu}^0+\frac{a_{\nu}^{(0)}}{a_{\nu}^*}\g^*(\hat{f}_{\nu}(x_{\nu})w),$$
whenever $w$ is in $W_{\nu}\oplus\bbq_{\nu}e_{\nu}^n$. Therefore
we have

$$\pi_{\nu}((H(\x)\mathcal{Y})_{\nu})
=(a_{\nu}+b_{\nu}\hat{f}_{\nu}(x_{\nu})x_{\nu}^{(0)})e_{\nu}^0\wedge
x_{\nu}^{(1)}\wedge\cdots\wedge
x_{\nu}^{(k-1)}+b_{\nu}x_{\nu}^{(0)}\wedge\cdots\wedge
x_{\nu}^{(k-1)}$$
$$\hspace*{10mm}+b_{\nu}\sum_{i=1}^{k-1}
\pm(\hat{f}_{\nu}(x_{\nu})x_{\nu}^{(i)})e_{\nu}^0\wedge\bigwedge_{s\neq
i}
x_{\nu}^{(s)}+b_{\nu}\frac{a_{\nu}^{(0)}}{a_{\nu}^*}\sum_{i=0}^{k-1}
\pm\g^*(\hat{f}_{\nu}(x_{\nu})x_{\nu}^{(i)})\wedge\bigwedge_{s\neq
i} x_{\nu}^{(s)}$$

\begin{equation}~\label{eq:norm}+\frac{a_{\nu}^{(0)}}{a_{\nu}^*}\sum_{i=1}^{k-1}\pm
\widetilde{\g}^*(\hat{f}_{\nu}(x_{\nu})x_{\nu}^{(i)},a_{\nu}+b_{\nu}\hat{f}_{\nu}(x_{\nu})x_{\nu}^{(0)})
\wedge e_{\nu}^0\wedge\bigwedge_{s\neq
0,i}x_{\nu}^{(s)}\hspace{7mm}\end{equation}

$$+b_{\nu}\frac{a_{\nu}^{(0)}}{a_{\nu}^*}\sum_{i,j=1, j>i}^{k-1}\pm
\widetilde{\g}^*(\hat{f}_{\nu}(x_{\nu})x_{\nu}^{(i)},\hat{f}_{\nu}(x_{\nu})x_{\nu}^{(j)})
\wedge e_{\nu}^0\wedge\bigwedge_{s\neq
i,j}x_{\nu}^{(s)}.\hspace{7mm}$$
By the choice of $x_{\nu}^{(i)}$'s, norm of the above vector would
be the maximum of norms of each of its summands. Using the fact
that maximum of a family of $(C_{\nu},\alpha_{\nu})$-good functions is again a
$(C_{\nu},\alpha_{\nu})$-good function, it suffices to show that the norm of
each of these summands is a $(C_{\nu},\alpha_{\nu})$-good function for a fixed
$C_{\nu}$ and $\alpha_{\nu}$. By theorem~\ref{linear}, we find a neighborhood
$V_{\nu}^1$ of $x_{\nu}$, $C_{\nu}^1$ and $\alpha_{\nu}^1>0$ such that the first
two lines would be $(C_{\nu}^1,\alpha_{\nu}^1)$-good functions on
$V_{\nu}^1$. Also, theorem~\ref{tartibat} provides us a
neighborhood $V_{\nu}^2$ of $x_{\nu}$, $C_{\nu}^2$, and $\alpha_{\nu}^2>0$ so
that the rest would be $(C_{\nu}^2,\alpha_{\nu}^2)$-good functions. 
Hence corollary 2.3 of~\cite{KT} gives us the claim.

\noindent
\textit{\textbf{Proof of (II).}} By looking at the first line of the
equation~(\ref{eq:norm}), one can see that
$\theta(\dbf\u_{\x}\Delta)\ge
\p\max\{|a_{\nu}+b_{\nu}\hat{f_{\nu}}(x_{\nu})\cdot x_{\nu}^{(0)}|,|b_{\nu}|\}$. Thus
$\theta(\dbf\u_{\x}\Delta)\le 1$ implies that
$\p\max\{|a_{\nu}|,|b_{\nu}|\}$ has an upper bound. Therefore using corollary 7.9 of \cite{KT}, we would get the finiteness of such $\Delta$'s, as we claimed.\vspace{1mm}

\noindent
\textit{\textbf{Proof of (III).}} Let $\mathbf{V}$ be the
neighborhood of $\x$ given by theorem~\ref{tartibat},
$\bbf\subseteq\mathbf{V}$ be a ball containing $\x, M, \rho_1,
\rho_2,$ and $\rho_3$ be as follows
$$\rho_1=\inf\{|f_{\nu}(x_{\nu})\cdot
Z_{\nu}+z_{\nu}^{0}|_{\nu}\hspace{1mm}|\hspace{1mm}\x\in\bbf, \nu\in S,
Z_{\nu}\in\bbq_{\nu}^n,
\para Z_{\nu}\para=1, z_{\nu}^{0}\in\bbq_{\nu}\},$$
$$\rho_2=\inf\{\sup_{\x\in\bbf}\para\g
f_{\nu}(x_{\nu})Z_{\nu}\para\hspace{1mm} |\nu\in S,
Z_{\nu}\in\bbq_{\nu}^n, \para Z_{\nu}\para=1\},\hspace{3cm}$$
 and $\rho_3$ is given
by theorem~\ref{tartibat}(a), and
$M=\sup_{\x\in\bbf}\max\{\para\fbf(\x)\para_S,\para\g\fbf(\x)\para_S\}.$

\noindent
If $\rank_{\bbz_S}\Delta=1$, then $\Delta$ can be represented by a
vector $\mathbf{w}=(w_{\nu})_{\nu\in S}$, with $w^{i}_{\nu}\in\mathbb{Z}_S$ for all $i$'s and for any $\nu\in S$. The first coordinate of $\dbf\u_{\x}\wbf$ is then equal to
$$(\frac{1}{a_{\nu}^{(0)}}(w_{\nu}^{(0)}+\sum_{i=1}^n
f_{\nu}^{(i)}(x_{\nu})w_{\nu}^{(i)}))_{\nu\in S}.$$ Therefore
$c(\dbf\u_{\x}\wbf)\ge\rho_1^{\kappa}$ since $|a_{\nu}^{(0)}|\le
1$.

Now assume $\rank_{\bbz_S} \Delta=k>1$. As in part (I), let us denote the $\mathbb{Q}_{\nu}$ span of the projection to $\nu$ place of $\Delta$ by $\Delta_{\nu}$. Let $x_{\nu}^{(1)},\cdots,x_{\nu}^{(k-2)}$ be an orthonormal set in $W_{\nu}\cap \Delta_{\nu}.$ We extend this to an orthonormal set in $(W_{\nu}\oplus \bbq_{\nu}e_{\nu}^n)\cap\Delta_{\nu}$ by adding $x_{\nu}^{(k-1)}$. Now if necessary choose a vector $x_{\nu}^{(0)}$ to complete $\{e_{\nu}^0,x_{\nu}^{(1)},\cdots,x_{\nu}^{(k-1)}\}$ to an orthonormal basis of $\Delta_{\nu}+\bbq_{\nu}e_{\nu}^0$.

Let $\{\ybf^{(1)},\cdots,\ybf^{(k)}\}$ be a $\bbz_S$-base of $\Delta$, and define $\mathcal{Y}=\ybf^{(1)}\wedge\cdots\wedge\ybf^{(k)}$. Since $D_{\nu}$ leaves $W_{\nu}, W_{\nu}^*, \bbq_{\nu} e_{\nu}^0$, and $\bbq_{\nu}e_{\nu}^n$ invariant, one has
$$\theta(\dbf\u_{\x}\Delta)=\theta(\dbf\u_{\x}\mathcal{Y})=\p\theta_{\nu} (D_{\nu}\u_{\x}^{\nu}\mathcal{Y}_{\nu})=\p \|D_{\nu}\pi_{\nu}(\u_{\x}^{\nu}\mathcal{Y}_{\nu})\|_{\nu}.$$
On the other hand, similar to the discussion in (I), there are $a_{\nu},b_{\nu}\in\bbq_{\nu}$ so that
$$\mathcal{Y}_{\nu}=a_{\nu} e_{\nu}^0\wedge
x_{\nu}^{(1)}\wedge\cdots\wedge
x_{\nu}^{(k-1)}+b_{\nu}x_{\nu}^{(0)}\wedge\cdots\wedge
x_{\nu}^{(k-1)}.$$
Note also that $\p\{|a_{\nu}|_{\nu},|b_{\nu}|_{\nu}\}\ge 1$. Similar to the argument of~\cite[Section 7]{BKM}, let $\check{\fbf}(\x)=(\cf_{\nu}(x_{\nu}))_{\nu\in S}$, where $$\cf_{\nu}(x_{\nu})=(1,0_{d_{\nu}},f_{\nu}^{(1)}(x_{\nu}),\cdots,f_{\nu}^{(n)}(x_{\nu})),$$ we would have:

$$\pi_{\nu}(\u^{\nu}_{\x}\mathcal{Y}_{\nu})
=(a_{\nu}+b_{\nu}\cf_{\nu}(x_{\nu})x_{\nu}^{(0)})e_{\nu}^0\wedge
x_{\nu}^{(1)}\wedge\cdots\wedge
x_{\nu}^{(k-1)}+b_{\nu}x_{\nu}^{(0)}\wedge\cdots\wedge
x_{\nu}^{(k-1)}$$
$$\hspace*{10mm}+b_{\nu}\sum_{i=1}^{k-1}
\pm(\cf_{\nu}(x_{\nu})x_{\nu}^{(i)})e_{\nu}^0\wedge\bigwedge_{s\neq
i}
x_{\nu}^{(s)}+b_{\nu}\sum_{i=0}^{k-1}
\pm\g^*(\cf_{\nu}(x_{\nu})x_{\nu}^{(i)})\wedge\bigwedge_{s\neq
i} x_{\nu}^{(s)}$$
$$+e_{\nu}^0\wedge \check{\mathcal{Y}}_{\nu}(x_{\nu}),\hspace{7cm}$$
$$\mbox{where}\hspace{2mm}\check{\mathcal{Y}}_{\nu}(x_{\nu})=\sum_{i=1}^{k-1}\pm
\widetilde{\g}^*(\cf_{\nu}(x_{\nu})x_{\nu}^{(i)},a_{\nu}+b_{\nu}\cf_{\nu}(x_{\nu})x_{\nu}^{(0)})
\wedge\bigwedge_{s\neq 0,i}x_{\nu}^{(s)}\hspace{7mm}$$
$$+b_{\nu}\sum_{i,j=1, j>i}^{k-1}\pm
\widetilde{\g}^*(\cf_{\nu}(x_{\nu})x_{\nu}^{(i)},\cf_{\nu}(x_{\nu})x_{\nu}^{(j)})
\wedge\bigwedge_{s\neq
i,j}x_{\nu}^{(s)}.\hspace{7mm}$$
In order to find a lower bound $\rho$ for $\hspace{1mm}\sup_{x\in\mathbf{B}} \theta(\mathbf{D}\mathcal{U}_{\mathbf{x}}\mathcal{Y}),$ it suffices to show that $\sup \p\|D_{\nu}\check{\mathcal{Y}}_{\nu}(x_{\nu})\|_{\nu}$ is not less that $\rho \p|a_{\nu}^{(0)}|_{\nu}.$ Now consider the product $\mathbf{e}_n\wedge \check{\mathcal{Y}}(\mathbf{x}).$ Our next task is to show: $$(*)\hspace{2mm}\sup\p\|e_{\nu}^n\wedge\check{\mathcal{Y}}_{\nu}(x_{\nu})\|_{\nu}\geq \rho.$$ Assume that $(*)$ holds and let us finish the proof. Since the eigenvalue with the smallest norm of $D_{\nu}$ on $W_{\nu}^{*}\wedge (\bigwedge^{k-1}(\bbq_{\nu} e_{\nu}^0\oplus W_{\nu}\oplus\bbq_{\nu}e_{\nu}^n))$ is equal to ${(a_{\nu}^{(*)}a_{\nu}^{(n-k+2)}\cdots a_{\nu}^{(n)})}^{-1},$ using $\|D_{\nu}(e_{\nu}^n\wedge\check{\mathcal{Y}}_{\nu}(x_{\nu}))\|_{\nu}\leq\|D_{\nu}\check{\mathcal{Y}}_{\nu}(x_{\nu})\|_{\nu}/|a_{\nu}^{(n)}|_{\nu},$ one has $$\p\|D_{\nu}\check{\mathcal{Y}}_{\nu}(x_{\nu})\|_{\nu}\geq\p|a_{\nu}^{
 (n)}|_{\nu}\|D_{\nu}(e_{\nu}^n\wedge\check{\mathcal{Y}}_{\nu}(x_{\nu}))\|_{\nu}$$
$$\geq\p\frac{|a_{\nu}^{(n)}|_{\nu}}{|a_{\nu}^{(*)}a_{\nu}^{(n-k+3)}\cdots a_{\nu}^{(n)}|_{\nu}}\|e_{\nu}^n\wedge\check{\mathcal{Y}}_{\nu}(x_{\nu})\|_{\nu}$$ $$\geq\rho\p\frac{|a_{\nu}^{(0)}|_{\nu}}{|a_{\nu}^{(0)}a_{\nu}^{(*)}a_{\nu}^{(n-k+3)}\cdots a_{\nu}^{(n-1)}|_{\nu}}\geq \rho\p|a_{\nu}^{(0)}|_{\nu},$$ as we wanted. Thus it suffices to show $(*).$ To that end for any place $\nu\in S$ select the term containing $x_{\nu}^{(1)}\wedge x_{\nu}^{(2)}\cdots\wedge x_{\nu}^{(k-2)},$ then one has $$e_{\nu}^n\wedge\check{\mathcal{Y}}_{\nu}(x_{\nu})=\pm z_{\nu}^{(*)}(x_{\nu})\wedge e_{\nu}^n\wedge x_{\nu}^{(1)}\wedge x_{\nu}^{(2)}\cdots\wedge x_{\nu}^{(k-2)}+\hspace{1mm}\begin{array}{l}\mbox{\rm{other terms where one}}\\ \mbox{\rm{or two}}\hspace{1mm} x_{\nu}^{(i)}\hspace{1mm}\mbox{\rm{are missing,}}\end{array}$$ where

$z_{\nu}^{(*)}(x_{\nu})={\widetilde{\nabla}}^{*}(\check{f}_{\nu}(x_{\nu})x_{\nu}^{k-1},\hspace{1mm}a_{\nu}+b_{\nu}\check{f}_{\nu}(x_{\nu})x_{\nu}^{(0)}) $ $$=b_{\nu}{\widetilde{\nabla}}^{*}(\check{f}_{\nu}(x_{\nu})x_{\nu}^{k-1},\hspace{1mm}\check{f}_{\nu}(x_{\nu})x_{\nu}^{(0)})-a_{\nu}{\nabla}^{*}(\check{f}_{\nu}(x_\nu)x_{\nu}^{(k-1)})$$ Using the first expression it follows that $\sup_{x_{\nu}\in B_{\nu}}\|z_{\nu}^{(*)}(x_{\nu})\|_{\nu}\geq \rho_3 \hspace{1mm}|b_{\nu}|_{\nu},$ where the second expression gives, $\hspace{1mm}\sup_{x_{\nu}\in B_{\nu}}\|z_{\nu}^{(*)}(x_{\nu})\|_{\nu}\geq \rho_2|a_{\nu}|_{\nu}-2M^2|b_{\nu}|_{\nu}.$  It is easy to see that there exists $\rho_0$
such that $$\max\{\rho_2|a_{\nu}|_{\nu}-2M^2|b_{\nu}|_{\nu},\rho_3 \hspace{1mm}|b_{\nu}|_{\nu}\}\geq \rho_0\cdot\max\{|a_{\nu}|_{\nu},|b_{\nu}|_{\nu}\}.$$ Therefore $\rho=\rho_0^{\kappa}$ satisfies the conditions of the theorem.

%%%%%%%%%%%%%%%%%%%%%%%%%%%%%%%%%%%%%%%%%%%%%%%%%%%%%%%%%%%%%%%%%%%%%%%%%%%%
%%%%%%%%%%%%%%%%%%%%%% Proof of the main theorem %%%%%%%%%%%%%%%%%%%%%%%%%%%
%%%%%%%%%%%%%%%%%%%%%%%%%%%%%%%%%%%%%%%%%%%%%%%%%%%%%%%%%%%%%%%%%%%%%%%%%%%%

\section{Proof of the main theorem}
Take $\x_0\in \mathbf{U}$. 
Choose a neighborhood $\mathbf{V}\subseteq
\mathbf{U}$ of $\x_0$ and a positive number $\alpha$, as in
theorem~\ref{<}, and pick a ball $\bbf=\p B_{\nu}\subseteq
\mathbf{V}$ containing $\mathbf{x}_0$ such that the ball with the
same center and triple the radius is contained in $\mathbf{U}$. We
prove that $\bbf\cap\mathcal{W}_{\mathcal{R},\Psi}^{\h\fbf}$ has measure
zero. For any $\mathbf{q}\in \mathcal{R}^n$, let
$$A_{\mathbf{q}}=\{(x_{\nu})_{\nu\in S}\in \bbf |\hspace{1mm}
|\langle(f_{\nu}(x_{\nu}))_{\nu\in S}\cdot
\mathbf{q}\rangle|<\Psi(\mathbf{q})\}.$$
We shall proceed by induction on $n$. For $n=1$ set $Q=\mathcal{R}\setminus\{0\}$ and for $n\ge 2$,
by the induction hypothesis, it would be enough to deal with the set $Q=\{(q_{\nu})_{\nu\in S}\in \mathcal{R}^n| q_{\nu}^{(i)}\neq
0\hspace{1mm}\mbox{for any}\hspace{1mm} \nu\hspace{1mm}
\&\hspace{1mm} i\}$ i.e. the set of
vectors with non-zero coordinates, namely we have to prove that the set of points $\mathbf{x}$ in $\bbf$ which belong to infinitely many
$A_\mathbf{q}$ for $\mathbf{q}\in Q$ has measure zero. Now let
$$A_{\ge
\mathbf{q}}=\left\{\mathbf{x}\in A_{\mathbf{q}}|\hspace{1mm}\begin{array}{l}\|\q\g\mathbf{f}(\x)\|_{\nu}>\|\q\|_{S}^{-\epsilon}\hspace{4mm}\nu\in {S_{\mathcal{R}}}^c\h \\ \|\q\g\mathbf{f}(\x)\|_{\nu}>\|\q\|_S^{1-\epsilon}\h\nu\in S_{\mathcal{R}}\end{array}\right\}\h \&$$\vspace{.5mm}
$$A_{<\mathbf{q}}=A_{\mathbf{q}}\setminus A_{\ge\mathbf{q}}.$$
For any $\mathbf{t}=(t_1,\cdots,t_n)\in \bbn^n$, let
$$\bar{A}_{\ge\mathbf{t}}=\bigcup_{\mathbf{q}\in Q, 2^{t_i}\le
|q^{(i)}|<2^{t_{i+1}}}A_{\ge\q}\hspace{2mm}\mbox{and}\hspace{2mm}\bar{A}_{<\mathbf{t}}=\bigcup_{\mathbf{q}\in Q, 2^{t_i}\le
|q^{(i)}|<2^{t_{i+1}}}A_{<\q}.$$ It is clear that the union of
$A_{\mathbf{q}}$'s where $\mathbf{q}$ varies in $Q$ is the
same as the union of the $\bar{A}_{\ge\mathbf{t}}$'s and $\bar{A}_{<\mathbf{t}}$'s where
$\mathbf{t}$ varies in $\bbn^n$.
\\

\noindent
By the conditions posed on $\Psi$, we have 
\begin{itemize}
\item[(i)] If for any $i$ one has $2^{t_i}\le |q^{(i)}|_S<2^{t_{i+1}}$,
 then $\Psi(\mathbf{q})\le\Psi(2^{t_1},\cdots,2^{t_n})$.
 \item[(ii)] For large enough $\para\mathbf{q}\para$, we have $\Psi(\mathbf{q})\le (\prod_i|q^{(i)}|_S)^{-g(\mathcal{R})}$.
 \end{itemize}
These show $\bar{A}_{\ge\mathbf{t}}$ is a subset of the set defined in theorem~\ref{>} \vspace{1mm}with $T_i=2^{t_i+1}$ and $\delta=2^{g(\mathcal{R})\sum_{i=1}^n(t_i+1)}\Psi(2^{t_1},\cdots,2^{t_n}).$ Now one notes that the convergence of the sum $\sum\Psi(\q)$ gives that of $\sum2^{g(\mathcal{R})\sum_{i=1}^n(t_i+1)}\Psi(2^{t_1},\cdots,2^{t_n})$. So Borel-Cantelli lemma gives us that almost all points of $\U$ are in at most finitely many $\bar{A}_{\ge\mathbf{t}},$ as we desired.

As we said
$\Psi(\mathbf{q})\le (\prod_i|q^{(i)}|_S)^{-g(\mathcal{R})}$
 for large enough $\para\mathbf{q}\para_S$. So if for
 any $i$ one has $2^{t_i}\le |q^{(i)}|_S<2^{t_{i+1}}$,
 then $\Psi(\mathbf{q})\le 2^{-g(\mathcal{R})\sum_{i=1}^n t_i}$, for large enough $\para\mathbf{t}\para$. Now for such $\mathbf{t}$ we may write $\bar{A}_{<\mathbf{t}}=\cup_{\nu\in S}\bar{A}_{\mathbf{<t},\nu}$ where each $\bar{A}_{<\mathbf{t},\nu}$ is contained
 in set defined in~\ref{<}, with $\delta=2^{\frac{-g(\mathcal{R})\sum_{i=1}^nt_i}{\kappa}},\hspace{1mm}T_i=2^{t_i+1}$ and
 $$\begin{array}{ll}K_{\nu}=2^{(1-\epsilon)({\para\mathbf{t}\para}+1)} \hspace{2mm}\mbox{if}\h\h\nu\in S_{\mathcal{R}},& \h K_{\nu}=2^{-\epsilon\para\mathbf{t}\para}\h\h\mbox{if}\h\h\nu\in {S_{\mathcal{R}}}^c \\ K_{\omega}=2^{\|t\|+1}\h\h\mbox{if}\h\h\omega\in S_{\mathcal{R}}\setminus\{\nu\} , & K_{\omega}=1 \hspace{2mm}\mbox{if}\h \omega\in {S_{\mathcal{R}}}^c\setminus\{\nu\}\end{array}$$ 
 It is not hard to verify the inequalities in the
 hypothesis of theorem~\ref{<}. Moreover, one has
 $$\vare^{\kappa{(n+1)}}=\max\{\delta^{\kappa(n+1)},\delta^{\kappa}(\frac{T_1\cdots T_n}{\max
 T_i})^{g(\mathcal{R})}\p K_{\nu}\}=\delta^{\kappa}(\frac{T_1\cdots T_n}{\max
 T_i})^{g(\mathcal{R})}\p K_{\nu},$$
So we have $\vare\leq C'2^{\frac{-\epsilon\para\mathbf{t}\para}{\kappa(n+1)}}$ for some constant $C'$ depending on $\mathbf{f}.$ So by theorem~\ref{<}, and the choice of $\mathbf{V}$ and $\bbf$, measure of
 $\bar A_{\mathbf{t}}$ is at most
 $$C 2^{-\frac{\alpha\epsilon\para\mathbf{t}\para}{(n+1)\kappa^2}}|\bbf|.$$
 Therefore the sum of measures of $\bar A_{<\mathbf{t}}$'s is
 finite, thus another use of Borel-Cantelli lemma completes the proof.

%%%%%%%%%%%%%%%%%%%%%%%%%%%%%%%%%%%%%%%%%%%%%%%%%%%%%%%%%%%%%%%%%%%%%%%%%%%
%%%%%%%%%%%%%%%%% Remarks and open problems %%%%%%%%%%%%%%%%%%%
%%%%%%%%%%%%%%%%%%%%%%%%%%%%%%%%%%%%%%%%%%%%%%%%%%%%%%%%%%%%%%%%%%%%%%%%%%%
\section{A few remarks and open problems}

\textbf{1.} In this article, we worked with product of non-degenerate $p$-adic analytic manifolds. Historically this is the case which has drawn most attention. However most of the argument is valid for the product of non-degenerate $C^k$ manifolds. The only part in which we use analyticity extensively is in the proof of lemma~\ref{bounded}. 

\textbf{2.} In this paper, we studied analytic manifolds containing a real analytic component.  In~\cite{MS}, we prove a convergence Khintchine-type theorem for simultaneous approximation in non-Archimedian places. There we also prove the divergent part in the $p$-adic case. The following is an important corollary of the main results of loc. cit.
\begin{thm}
Let $M\subseteq\bbq_p^n$ be a $p$-adic non-degenerate analytic manifold. Suppose $\Psi:\bbz^n\setminus\{0\}\rightarrow(0,\infty)$ is a function of norm and decreasing in terms of norm. Then almost every (resp. almost no) point of $M$ is $\Psi$-Approximable if $\sum_{\q\in\bbz^n}\Psi(\q)=\infty$ (resp $\sum_{\q\in\bbz^n}\Psi(\q)<\infty$).
\end{thm}

\textbf{3.} Both here and in~\cite{MS}, we consider {\it homogeneous diophantine approximation}, namely we are approximating zero. One can consider the inhomogeneous problem.  As we mentioned in the introduction, V.~Bernik and E.~Kovalevskaya~\cite{BK2} proved the inhomogeneous problem for the Veronese curve in product of local fields, i.e. $\bbc\times\bbr\times\prod_{p\in S} \bbq_p$. It would be interesting if the inhomogeneous problem could be proved for non-degenerate manifolds.

\textbf{4.} As we recalled in the introduction, historically there are two kinds
of Diophantine approximations. One of them is coming from the dot product which
is the question that we considered, and the other one is simultaneous 
approximation of each of the components. 

\begin{problem}~\label{dual}
Let $\vec{f}=(f_1,\cdots,f_n)$, where $f_i$'s are analytic functions from an open 
subset $U$ of $\bbr^d$ to $\bbr$ and $1,f_1,\cdots,f_n$ are linearly 
independent. Let $\psi$ be a decreasing map from $\bbz$ to $\bbr^+$. Define
$$\mathcal{W}_{f,\psi}=\{\mathbf{x}\in \mathbf{U}| \|q \vec{f}(\mathbf{x})+\vec{p}\|<\psi(q)\hspace{1mm}\mbox{for infinitely many}\hspace{1mm} q\in\bbz\hspace{1mm}\mbox{and} \hspace{1mm}\vec{p}\in\bbz^n\}.$$
Then $\mathcal{W}_{f,\psi}$ is null (resp. co-null) if $\sum_{q\in\bbz}\psi(q)^n$ is convergent (resp. divergent).
\end{problem}
\noindent
For a general $\psi$ very little is known. However there are partial results in this
direction, e.g. Dodson, Rynne, Vickers~\cite{DRV} proved the convergence Khintchine-type
theorem for a non-degenerate manifold $M$ which is 2-convex at almost every point
i.e. at almost every point $\xi$ for any unit vector $v\in T_{\xi}M$, at least 
two of the principal curvatures $\tau_i(\xi,v)$, are non-zero and have the same sign.
Much more is known for the case of planar curves ref.~\cite{BDVV}, where they settled the 
divergence case for $C^3$-planner curves and the convergence case for rational quadratic curves.
However even the case of the curve $(x,x^2,x^3)$ is still open.

%%%%%%%%%%%%%%%%%%%%%%%%%%%%%%%%%%%%%%%%%%%%%%%%%%%%%%%%%%%%%%%%%%%%%%%%%%%
%%%%%%%%%%%%%%%%%%%%%%%%%%%%%% Bibliography  %%%%%%%%%%%%%%%%%%%%%%%%%%%%%%
%%%%%%%%%%%%%%%%%%%%%%%%%%%%%%%%%%%%%%%%%%%%%%%%%%%%%%%%%%%%%%%%%%%%%%%%%%%

%%%%%%%%%%%%%%%%%%%%%%%%%%%%%%%%%%%%%%%%%%%%%%%%%%%%%%%%%%%%%%%%%%%%%%%%%%%%%

{\sc Dept. of Math., Yale Univ., New Haven, CT, 06520}

{\em E-mail address:} {\tt amir.mohammadi@yale.edu}

{\sc Dept. of Math, Princeton Univ., Princeton, NJ, 08544}

{\em E-mail address:} {\tt asalehi@Math.Princeton.EDU}

\end{document}